\documentclass[11pt]{amsart}

\usepackage{hyperref}
\usepackage{amsmath,latexsym,amssymb,amsthm,array,amsfonts}

\usepackage{mathrsfs,dsfont,bbm}

\numberwithin{equation}{section}

\newtheorem{definition}{Definition}[section]
\newtheorem{theorem}{Theorem}[section]
\newtheorem{lemma}{Lemma}[section]
\newtheorem{corollary}{Corollary}[section]
\newtheorem{proposition}{Proposition}[section]

\begin{document}

\title[Quadratic covariations for the solution to a SHE]
{Quadratic covariations for the solution to a stochastic heat equation${}^{*}$}

\footnote[0]{${}^{*}$The Project-sponsored by NSFC (No. 11571071, 11426036), and Innovation Program of Shanghai Municipal Education Commission (No. 12ZZ063)}

\author[X. Sun, L. Yan and X. Yu]{Xichao Sun${}^{1}$, Litan Yan${}^{2,\S}$ and Xianye Yu${}^{2}$}

\footnote[0]{${}^{\S}$litan-yan@hotmail.com (Corresponding Author)}


\date{}



\maketitle

\begin{center}
{\footnotesize {\it ${}^1$Department of Mathematics and Physics, Bengbu University\\
1866 Caoshan Rd., Bengbu 233030, P.R. China\\
${}^2$Department of Mathematics, College of Science, Donghua University\\
2999 North Renmin Rd., Songjiang, Shanghai 201620, P.R. China}}
\end{center}

\maketitle


\begin{abstract}
Let $u(t,x)$ be the solution to a stochastic heat equation
$$
\frac{\partial}{\partial t}u=\frac12\frac{\partial^2}{\partial x^2}u+\frac{\partial^2}{\partial t\partial x}X(t,x),\quad t\geq 0, x\in {\mathbb R}
$$
with initial condition $u(0,x)\equiv 0$, where $X$ is a time-space white noise. This paper is an attempt to study stochastic analysis questions of the solution $u(t,x)$. In fact, the solution is a Gaussian process such that the process $t\mapsto u(t,\cdot)$ is a bi-fractional Brownian motion seemed a fractional Brownian motion with Hurst index $H=\frac14$ for every real number $x$. However, the properties of the process $x\mapsto u(\cdot,x)$ are unknown. In this paper we consider the quadratic covariations of the two processes $x\mapsto u(\cdot,x),t\mapsto u(t,\cdot)$. We show that $x\mapsto u(\cdot,x)$ admits a nontrivial finite quadratic variation and the forward integral of some adapted processes with respect to it coincides with "It\^o's integral", but it is not a semimartingale. Moreover, some generalized It\^o's formulas and Bouleau-Yor identities are introduced.
\end{abstract}

\tableofcontents

\section{Introduction}\label{sec1}
Let $u(t,x)$ denote the solution to the stochastic heat equation
\begin{equation}\label{sec1-eq1.-1}
\frac{\partial}{\partial
t}u=\frac12\frac{\partial^2}{\partial x^2}u+\frac{\partial^2}{\partial
t\partial x}X(t,x),\quad t\geq 0, x\in {\mathbb R}
\end{equation}
with initial condition $u(0,x)\equiv 0$, where $\dot{X}$ is a
time-space white noise on $[0,\infty)\times {\mathbb R}$. That is,
$$
u(t,x)=\int_0^t\int_{\mathbb{R}}p(t-r,x-y)X(dr,dy),
$$
where $p(t,x)=\frac1{\sqrt{2\pi t}}e^{-\frac{x^2}{2t}}$ is the heat kernel. Then, these processes $(t,x)\mapsto u(t,x)$, $t\mapsto u(t,\cdot)$ and $x\mapsto u(\cdot,x)$ are Gaussian. Swanson~\cite{Swanson} has showed that
\begin{equation}\label{sec1-eq1.0}
E\left[u(t,x)u(s,x)\right]=\frac1{\sqrt{2\pi}}
\left((t+s)^{1/2}-|t-s|^{1/2}\right),\qquad t,s\geq 0,
\end{equation}
and the process $t\mapsto u(t,x)$ has a nontrivial quartic
variation. This shows that for every $x\in {\mathbb R}$, the process $t\mapsto u(t,x)$ coincides with the bi-fractional Brownian motion and it is not a semimartingale, so a stochastic integral with respect to the process $t\mapsto u(t,x)$ cannot be defined in the classical It\^o sense. Some surveys and complete literatures for bi-fractional Brownian motion could be found in Houdr\'e and Villa~\cite{Hou}, Kruk {\em et al}~\cite{Kruk}, Lei and Nualart~\cite{Lei-Nualart}, Russo and Tudor~\cite{Russo-Tudor}, Tudor and Xiao~\cite{Tudor-Xiao} and Yan {\em et al}~\cite{Yan4}, and the references therein. It is important to note for a large class of parabolic SPDEs, one obtains better regularity results when the solution $u$ is viewed as a process $t\mapsto u(t,x)$ taking values in Sobolev space, rather than for each fixed $x$. Denis~\cite{Denis} and Krylov~\cite{Krylov} considered a class of stochastic partial differential equations driven by a multidimensional Brownian motion and showed that the solution is a Dirichlet processes. These inspire one to consider stochastic calculus with respect to the solution to the stochastic heat equation. It is well known that many authors have studied some It\^o analysis questions of the solutions of some stochastic partial differential equations and introduced the related It\^o and Tanaka formula (see, for examples, Da Prato {\em et al}~\cite{Da Prato}, Deya and Tindel~\cite{Deya-Tindel}, Gradinaru {\em et al}~\cite{Grad5}, Lanconelli~\cite{Lanconelli1,Lanconelli2}, Le\'on and Tindel {\em et al}~\cite{Leon-Tindel}, Nualart and Vuillermot~\cite{Nua5}, Ouahhabi and Tudor~\cite{Ouahhabi-Tudor}, Pardoux~\cite{Pardoux}, Torres {\em et al}~\cite{Torres}, Tudor~\cite{Tudor}, Tudor and Xiao~\cite{Tudor-Xiao2}, Zambotti~\cite{Zambotti}, and the references therein). Almost all of these studies considered only the process in time, and there is a little discussion about the process $x\mapsto u(\cdot,x)$. This paper is an attempt to study stochastic analysis questions of the solution $u(t,x)$.

On the other hand, we shall see (in Section~\ref{sec4}) that
the process $x\mapsto u(\cdot,x)$ admits a nontrivial finite quadratic variation coinciding with the classical Brownian motion in any finite interval, and moreover we shall also see (in Section~\ref{sec4-1}) that the forward integral of some adapted processes with respect to $x\mapsto u(\cdot,x)$ coincides with "It\^o's integral". As a noise, the stochastic process $u=\{u(t,x), t\geq 0,x\in {\mathbb R}\}$ is {\em very rough} in time and it is not white in space. However, the process $x\mapsto u(\cdot,x)$ admits some characteristics similar to Brownian motion. These results, together with the works of Swanson~\cite{Swanson}, point out that the process $u=\{u(t,x)\}$ as a noise admits the next special structures:
\begin{itemize}
\item It is very rough in time and similar to fractional Brownian motion with Hurst index $H=\frac14$, but it has not stationary increments.
\item It is not white in space, but its quadratic variation coincides with the classical Brownian motion and it is not self-similar.
\item The process in space variable is not a semimartingale, but the forward integral of some adapted processes with respect to the process in space variable coincides with "It\^o's integral".
\item The process $u=\{u(t,x)\}$ admits a simple representation via Wiener integral with respect to Brownian sheet.
\item Though the process $u=\{u(t,x)\}$ is Gaussian, as a noise, its time and space parts are farraginous. We can not decompose its covariance as the product of two independent parts. This is very different from fractional noise and white noise. In fact, we have
    \begin{align*}
Eu(t,x)u(s,y)=\frac1{\sqrt{2\pi}}\int_0^s\frac1{\sqrt{t+s-2r}}
\exp\left\{-\frac{(x-y)^2}{2(t+s-2r)}\right\}dr
\end{align*}
for all $t\geq s>0$ and $x,y\in {\mathbb R}$.
\end{itemize}
Therefore, it seems interesting to study the integrals
$$
\int_{\mathbb R}f(x)u(t,dx),\quad\int_0^tf(s)u(ds,x),\quad\int_0^t\int_{\mathbb R}f(s,x)u(ds,dx),
$$
and some related stochastic (partial) differential equations. For example, one can consider the following "iterated" stochastic partial differential equations:
$$
\frac{\partial}{\partial t}u^j=\frac12\frac{\partial^2}{\partial x^2}u^j+f(u^j)+\frac{\partial^2}{\partial t\partial x}u^{j-1}(t,x),\quad t\geq 0, x\in {\mathbb R},\quad j=1,2,\ldots,
$$
where $u^0$ is a time-space white noise. Of course, one can also consider some sample path properties and singular integrals associated with the solution process $u=\{u(t,x),t\geq 0,x\in {\mathbb R}\}$. We will carry out these projects in some forthcoming works. In the present paper our start points are to study the quadratic variations of the two processes $x\mapsto u(\cdot,x),t\mapsto u(t,\cdot)$. Our objects are to study the quadratic covariations of $x\mapsto u(\cdot,x)$ and $t\mapsto u(t,\cdot)$, and moreover, we shall also introduce some generalized It\^o formulas associated with $\{u(\cdot,x),x\in {\mathbb R}\}$ and $\{u(t,\cdot),t\geq 0\}$, respectively, and to consider their local times and Bouleau-Yor's identities.

To expound our aim, let us start with a basic definition. An elementary calculation can show that (see Section~\ref{sec2})
\begin{equation}\label{sec1-eq1.1}
E[(u(t,x)-u(s,y))^2]= \frac1{\sqrt{2\pi}}\left(\sqrt{2|t-s|}+\Delta(s,t,x-y)\right)
\end{equation}
for all $t,s>0$ and $x,y\in {\mathbb R}$, where
$$
\Delta(s,t,z)=\int_0^{s\wedge t}\left(
\frac1{\sqrt{2(t-r)}}-\frac{2}{\sqrt{t+s-2r}}
\exp\left\{-\frac{z^2}{2(t+s-2r)}\right\}+\frac1{\sqrt{2(s-r)}} \right)dr
$$
for $t,s>0$ and $z\in {\mathbb R}$. This simple estimate inspires us to consider the following limits:
\begin{equation}
\begin{split}
\lim_{\varepsilon\to 0}\frac1{\varepsilon}E[(u(t,x+\varepsilon)&-u(t,x))^2]
=\frac1{\sqrt{\pi}}\lim_{\varepsilon\to 0}\frac1{\varepsilon}\int_0^{t}
\frac1{\sqrt{r}}\left(1-e^{-\frac{\varepsilon^2}{4r}}\right)dr\\
&=\frac2{\sqrt{\pi}}\int_0^{\infty}
\frac1{s^2}\left(1-e^{-\frac{s^2}{4}}\right)ds=1
\end{split}
\end{equation}
and
\begin{equation}
\lim_{\varepsilon\to 0}\frac1{\sqrt{\varepsilon}} E[(u(t+\varepsilon,x)-u(t,x))^2] =\sqrt{\frac2{\pi}}
\end{equation}
for all $t\geq 0$ and $x\in {\mathbb R}$. That is,
$$
\lim_{\delta\to 0}\lim_{\varepsilon\to 0}\frac1{\sqrt{\varepsilon}+\delta} E[(u(t+\varepsilon,x+\delta)-u(t,x))^2]=1,
$$
$$
\lim_{\varepsilon\to 0}\lim_{\delta\to 0}\frac1{\sqrt{\varepsilon}+\delta} E[(u(t+\varepsilon,x+\delta)-u(t,x))^2]=\sqrt{\frac2{\pi}}
$$
for all $t\geq 0$ and $x\in {\mathbb R}$. However, it is easy  to see that the limit
\begin{align*}
\lim\limits_{\substack{\varepsilon\to 0\\ \delta\to 0}}\frac1{\sqrt{\varepsilon}+\delta} E[(u(t+\varepsilon,x+\delta)-u(t,x))^2]
\end{align*}
does not exist for all $t>0, x\in {\mathbb R}$ by taking $\varepsilon=k\delta^2$ with $k>0$. Thus, the next definition is natural.
\begin{definition}
Denote $B:=\{B_t:=u(t,\cdot),t\geq 0\}$ and $W:=\{W_x:=u(\cdot,x),x\in {\mathbb R}\}$. Let $I_x=[0,x]$ for $x\geq 0$ and $I_x=[x,0]$ for $x\leq 0$. Define the integrals
\begin{align*}
I_\delta^1(f,x,t)&=\frac1{\delta}\int_{I_x} \left\{f(W_{y+\delta})-f(W_y)\right\}(W_{y+\delta}-W_y)dy,\\
I_\varepsilon^2(f,x,t)&=\frac1{\sqrt{\varepsilon}}\int_0^t \left\{f(B_{s+\varepsilon})-f(B_s)\right\}
(B_{s+\varepsilon}-B_s)\frac{ds}{2\sqrt{s}}
\end{align*}
for all $t\geq 0,x\in {\mathbb R},\varepsilon,\delta>0$, where $f$ is a measurable function on ${\mathbb R}$. The limits $\lim\limits_{\delta\to 0}I_\delta^1(f,t,x)$ and $\lim\limits_{\varepsilon\to 0}I_\varepsilon^2(f,t,x)$ are called the partial quadratic covariations (PQC, in short) in space and in time, respectively, of $f(u)$ and $u$, provided these limits exist in probability. We denote them by $[f(W),W]^{(SQ)}_x$ and $[f(B),B]^{(TQ)}_t$, respectively.
\end{definition}
Clearly, we have (see Section~\ref{sec4})
$$
[f(W),W]^{(SQ)}_x=\int_{I_x}f'(W_y)dy
$$
and $[W,W]^{(SQ)}_x=|x|$ for all $f\in C^1({\mathbb R}),t>0,x\in {\mathbb R}$. We also have (see Section~\ref{sec6})
$$
[f(B),B]^{(TQ)}_t=\int_0^tf'(B_s)\frac{ds}{\sqrt{2\pi s}}
$$
and $[B,B]^{(TQ)}_t=\sqrt{\frac{2}{\pi}t}$ for all $f\in C^1({\mathbb R}),t\geq 0,x\in {\mathbb R}$. These say that the process $W=\{W_x=u(\cdot,x),x\in {\mathbb R}\}$ admits a nontrivial finite quadratic variation in any finite interval $I_x$. This is also a main motivation to study the solution of~\eqref{sec1-eq1.-1}.

This paper is organized as follows. In Section~\ref{sec2}, we establish some technical estimates associated with the solution, and as some applications we introduce Wiener integrals with respect to the two processes $B=\{B_t=u(t,\cdot),t\geq 0\}$ and $W=\{W_x=u(\cdot,x),x\in {\mathbb R}\}$, respectively. In Section~\ref{sec4} we show that the quadratic variation $[W,W]^{(SQ)}$ exists in $L^2(\Omega)$ and equals to $|x|$ in every finite interval $I_x$. For a given $t>0$, by estimating in $L^2$
$$
\frac1{\varepsilon}\int_{I_x}f(W_{y+\varepsilon})
(W_{y+\varepsilon}-W_{y})dy\quad (x\in {\mathbb R})
$$
and
$$
\frac1{\varepsilon}\int_{I_x}f(W_y)(W_{y+\varepsilon}-W_y)dy
\quad (x\in {\mathbb R})
$$
for all $\varepsilon>0$, respectively, we construct a Banach space ${\mathscr H}_t$ of measurable functions such that the PQC $[f(W),W]^{(SQ)}$ in space exists in $L^2(\Omega)$ for all $f\in {\mathscr H}_t$, and in particular we have
$$
[f(W),W]_x^{(SQ)}=\int_{I_x}f'(W_y)dy
$$
provided $f\in C^1({\mathbb R})$. In Section~\ref{sec4-1}, as an application of Section~\ref{sec4}, we show that the It\^o's formula
\begin{align*}
F(W_x)=F(W_0)+\int_{I_x}f(W_y)\delta W_y+\frac1{2}[f(W),W]_x^{(SQ)}
\end{align*}
holds for all $t>0,x\in {\mathbb R}$, where the integral  $\int_{I_x}f'(W_y)\delta W_y$ denotes the Skorohod integral, $F$ is an absolutely continuous function with the derivative $F'=f\in {\mathscr H}_t$. In order to show that the above It\^o formula we first introduce a standard It\^o type formula
\begin{equation}\label{sec1-1-eq4.80011}
F(W_x)=F(W_0)+\int_{I_x}F'(W_y)\delta W_y +\frac1{2}\int_{I_x}F''(W_y)dy
\end{equation}
for all $F\in C^2({\mathbb R})$ satisfying some suitable  conditions. It is important to note that the Gaussian process
$W=\{W_x=u(\cdot,x),x\in {\mathbb R}\}$ does not satisfy the condition in Al\'os {\em et al}~\cite{Nua1} since
\begin{align*}
E\left[u(t,x)^2\right]=\sqrt{\frac{t}{\pi}},\quad \frac{d}{x}E\left[u(t,x)^2\right]=0
\end{align*}
for all $t\geq 0$ and $x\in {\mathbb R}$. We need to give the proof of the formula~\eqref{sec1-1-eq4.80011}. Moreover, we also show that the forward integral (see Russo-Vallois~\cite{Russo-Vallois2,Russo-Vallois3})
$$
\int_{I_x}f(W_y)d^{-}W_y:={\rm ucp}\lim_{\varepsilon\downarrow 0}\frac1{\varepsilon}\int_{I_x}f(W_y) \left(W_{y+\varepsilon}-W_y\right)dy
$$
coincides with the Skorohod integral $\int_{I_x}f(W_y)\delta W_y$, if $f$ satisfies the growth condition
\begin{equation}
|f(y)|\leq Ce^{\beta {y^2}},\quad y\in {\mathbb R}
\end{equation}
with $0\leq \beta<\frac{\sqrt{\pi}}{4\sqrt{t}}$, where the notation ${\rm ucp}\lim$ denotes the uniform convergence in probability on each compact interval. This is very similar to Brownian motion, but the process $W=\{W_x=u(\cdot,x),x\in {\mathbb R}\}$ is not a semimartingale. In Section~\ref{sec4-2} we consider some questions associated with the local time
$$
{\mathscr L}^t(x,a)=\int_0^x\delta(W_y-a)dy
$$
of the process $W=\{W_x=u(\cdot,x),x\geq 0\}$. In particular, we show that the Bouleau-Yor type identity
$$
[f(W),W]^{(SQ)}_x=-\int_{\mathbb {R}}f(v){\mathscr L}^t(x,dv)
$$
holds for all $f\in {\mathscr H}_t$. In Section~\ref{sec6} we consider some analysis questions associated with the quadratic covariation of the process $B=\{B_t=u(t,\cdot),t\geq 0\}$.

\section{Some basic estimates and divergence integrals}\label{sec2}
In this section we will establish divergence integral and some technical estimates associated with the solution
$$
u(t,x)=\int_0^t\int_{\mathbb{R}}p(t-r,x-y)X(dr, dy),\quad t\geq 0,x\in {\mathbb R},
$$
where $p(t,x)=\frac1{\sqrt{2\pi t}}e^{-\frac{x^2}{2t}}$ is the heat kernel. For simplicity throughout this paper we let $C$ stand for a positive constant depending only on the subscripts and its value may be different in different appearance, and this assumption is also adaptable to $c$. Moreover, we assume that the notation $F\asymp G$ means that there are positive constants $c_1$ and $c_2$ such that
$$
c_1G(x)\leq F(x)\leq c_2G(x)
$$
in the common domain of definition for $F$ and $G$.

The first object in this section is to introduce some basic estimates for the solution process $\{u(t,x),t\geq 0,x\in {\mathbb R}\}$. We have
\begin{align*}
R_{x,y}(s,t):&=Eu(t,x)u(s,y)=\int_0^s\int_{\mathbb{R}} p(t-r,x-z)p(s-r,y-z)dzdr\\
&=\frac1{2\pi}
\int_0^s\int_{\mathbb{R}}\frac1{\sqrt{(t-r)(s-r)}}
\exp\left\{-\frac{(x-z)^2}{2(t-r)}-\frac{(y-z)^2}{2(s-r)}\right\}
dzdr\\
&=\frac1{\sqrt{2\pi}}\int_0^s\frac1{\sqrt{t+s-2r}}
\exp\left\{-\frac{(x-y)^2}{2(t+s-2r)}\right\}dr
\end{align*}
for all $t\geq s>0$ and $x,y\in {\mathbb R}$. Denote
$$
\Delta(s,t,u)=\int_0^s\left(
\frac1{\sqrt{2(t-r)}}-\frac{2}{\sqrt{t+s-2r}}
\exp\left\{-\frac{u^2}{2(t+s-2r)}\right\}+\frac1{\sqrt{2(s-r)}} \right)dr
$$
for $t\geq s$. Then, we have
\begin{equation}\label{sec2-eq2.1}
E[(u(t,x)-u(s,y))^2]= \frac1{\sqrt{2\pi}}\left(\sqrt{2(t-s)}+\Delta(s,t,x-y)\right)
\end{equation}
for all $t>s>0$ and $x,y\in {\mathbb R}$.
\begin{lemma}\label{lem2.1}
For all $t\geq s>0$ and $x,y\in {\mathbb R}$ we have
\begin{equation}\label{sec2-eq2.2}
E\left[(u(t,x)-u(s,y))^2\right]\leq C\left(\sqrt{t-s}+|x-y|\right).
\end{equation}
\end{lemma}
\begin{proof}
Notice that
\begin{align*}
\int_0^s\frac{2}{\sqrt{t+s-2r}}&\left(
1-\exp\left\{-\frac{u^2}{2(t+s-2r)}\right\}\right)dr\\ &=|u|\int_{\frac{|u|}{\sqrt{t+s}}}^{\frac{|u|}{\sqrt{t-s}}}
\left(1-e^{-\frac{r^2}{2}}\right)\frac{dr}{r^2}\leq
|u|\int_0^{+\infty}
\left(1-e^{-\frac{r^2}{2}}\right)\frac{dr}{r^2}=|u|\sqrt{\frac{\pi}2}
\end{align*}
for all $t\geq s>0$ and $u\in {\mathbb R}$. We get
\begin{align*}
\Delta(s,t,u)&=\int_0^s\left(
\frac1{\sqrt{2(t-r)}}-\frac{2}{\sqrt{t+s-2r}}
+\frac1{\sqrt{2(s-r)}}\right)dr\\
&\qquad+\int_0^s\frac{2}{\sqrt{t+s-2r}}\left(
1-\exp\left\{-\frac{u^2}{2(t+s-2r)}\right\}\right)dr\\
&=\sqrt{2t}+\sqrt{2s}+(2-\sqrt{2})\sqrt{t-s}-2\sqrt{t+s}\\
&\qquad+\int_0^s\frac{2}{\sqrt{t+s-2r}}\left(
1-\exp\left\{-\frac{u^2}{2(t+s-2r)}\right\}\right)dr\\
&\leq C\left((3-\sqrt{2})\sqrt{t-s}+|u|\right)
\end{align*}
for all $t\geq s>0,u\in {\mathbb R}$ by the next estimate:
\begin{align*}
0\leq \sqrt{2t}+\sqrt{2s}&+2\sqrt{t-s}-2\sqrt{t+s}\\
&=2\sqrt{t-s} +\left(\sqrt{2t}-\sqrt{t+s}\right)+\sqrt{2s}-\sqrt{t+s}\leq 3\sqrt{t-s}.
\end{align*}
It follows from~\eqref{sec2-eq2.1} that
\begin{equation}\label{sec2-eq2.4}
E[(u(t,x)-u(s,y))^2]\leq C\left(\sqrt{t-s}+|x-y|\right)
\end{equation}
for all $t\geq s>0,x,y\in {\mathbb R}$. This completes the proof.
\end{proof}

\begin{lemma}
For all $t,s,r>0$ and $x\in {\mathbb R}$ we have
\begin{align*}
|E\left[u(r,x)(u(t,x)-u(s,x))\right]|&\leq C\sqrt{|t-s|}.
\end{align*}
\end{lemma}
\begin{proof}
For all $t,s,r>0$ and $x\in {\mathbb R}$, we have
\begin{align*}
E[u(r,x)(u(t,x)&-u(s,x))]=Eu(r,x)u(t,x)-Eu(r,x)u(s,x)\\
&=R_{x,x}(r,t)-R_{x,x}(r,s)\\
&=\frac1{\sqrt{2\pi}}\left(\sqrt{t+r}-\sqrt{|t-r|}-\sqrt{s+r} +\sqrt{|s-r|}\right),
\end{align*}
which gives
\begin{align*}
|E[u(r,x)&(u(t,x)-u(s,x))]|\\
&\leq \left|\sqrt{t+r}-\sqrt{s+r}\right| +\left|\sqrt{|t-r|}-\sqrt{|s-r|}\right|\leq 3\sqrt{|t-s|}
\end{align*}
for all $t,s,r>0$ and $x\in {\mathbb R}$.
\end{proof}
\begin{lemma}\label{lem2.3}
For all $t>0$ and $x,y,z\in {\mathbb R}$ we have
\begin{align*}
|E\left[u(t,x)(u(t,y)-u(t,z))\right]|&\leq C|y-z|.
\end{align*}
\end{lemma}
\begin{proof}
For all $t>0$ and $x,y,z\in {\mathbb R}$, we have
\begin{align*}
E[&u(t,x)(u(t,y)-u(t,z))]=Eu(t,x)u(t,y)-Eu(t,x)u(t,z)\\
&=R_{x,y}(t,t)-R_{x,z}(t,t)\\
&=\frac1{2\sqrt{\pi}}\left(\int_0^t\frac1{\sqrt{t-r}}
\exp\left\{-\frac{(x-y)^2}{4(t-r)}\right\}dr
-\int_0^t\frac1{\sqrt{t-r}} \exp\left\{-\frac{(x-z)^2}{4(t-r)}\right\}dr\right)\\
&=\frac{\sqrt{t}}{2\sqrt{\pi}}\left(|x-y| \int_{|x-y|}^{+\infty}\frac1{s^2}
e^{-\frac{s^2}{4t}}ds
-|x-z|\int_{|x-z|}^{+\infty}\frac1{s^2}
e^{-\frac{s^2}{4t}}ds\right).
\end{align*}
Consider the function $f:{\mathbb R}_{+}\to {\mathbb R}_{+}$ defined by
$$
f(x)=x\int_x^{+\infty}\frac1{s^2}e^{-\frac{s^2}{4t}}ds
=e^{-\frac{x^2}{4t}}-
\frac{x}{2t}\int_x^\infty e^{-\frac{r^2}{4t}}dr
$$
Then, by Mean value theorem we have
\begin{align*}
\left|f(u)-f(v)\right|& =\frac1{2t}|u-v|\int_\xi^{+\infty}e^{-\frac{s^2}{4t}}ds\\
&\leq \frac1{2t}|u-v| \int_0^{+\infty}e^{-\frac{s^2}{4t}}ds
\leq \frac{\sqrt{\pi}}{2\sqrt{t}}|u-v|
\end{align*}
for all $u,v\geq 0$ and some $\xi$ between $u$ and $v$. It follows that
$$
|E[u(t,x)(u(t,y)-u(t,z))]|\leq \frac14||x-y|-|x-z||\leq \frac14|y-z|
$$
for all $t>0$ and $x,y,z\in {\mathbb R}$.
\end{proof}
\begin{lemma}\label{lem2.4}
For all $t>s>t'>s'>0$ and $x\in {\mathbb R}$ we have
\begin{equation}\label{sec2-eq2.6}
|E\left[(u(t,x)-u(s,x))(u(t',x)-u(s',x))\right]|
\leq \frac{C(t'-s')\sqrt{t-s}}{\sqrt{ts(s-s')(t-t')}}
\end{equation}
\end{lemma}
\begin{proof}
For all $t>s>t'>s'>0$ and $x\in {\mathbb R}$ we have
\begin{align*}
E[(u(t,x)&-u(s,x))(u(t',x)-u(s',x))]\\ &=R_{x,x}(t,t')-R_{x,x}(s,t')-R_{x,x}(t,s')+R_{x,x}(s,s')\\
&=\frac1{\sqrt{2\pi}}\left(
\sqrt{t+t'}-\sqrt{t-t'}-\sqrt{s+t'}+\sqrt{s-t'}\right.\\
&\qquad\qquad\left.-\sqrt{t+s'}+\sqrt{t-s'} +\sqrt{s+s'}-\sqrt{s-s'}\right).
\end{align*}
Consider the function
$$
f(x)=\sqrt{t+x}-\sqrt{t-x}-\sqrt{s+x}+\sqrt{s-x}
$$
with $x\in [0,s]$. Then, we have
\begin{align*}
E[(u(t,x)-u(s,x))(u(t',x)-u(s',x))] =\frac1{\sqrt{2\pi}}\left(f(t')-f(s')\right),
\end{align*}
and by Mean value theorem
\begin{align*}
\left|f(t')-f(s')\right|&=\frac12(t'-s')
\left|\frac1{\sqrt{t+\xi}}-\frac1{\sqrt{t-\xi}}
-\frac1{\sqrt{s+\xi}}+\frac1{\sqrt{s-\xi}}\right|\\
&\leq \frac12(t'-s')\left(\frac{\sqrt{t+\xi}-\sqrt{s+\xi}}{
\sqrt{t+\xi}\sqrt{s+\xi}}+\frac{\sqrt{t-\xi}-\sqrt{s-\xi}}{ \sqrt{s-\xi}\sqrt{t-\xi}}\right)\\
&\leq \frac{C(t'-s')\sqrt{t-s}}{\sqrt{ts(s-s')(t-t')}}
\end{align*}
for some $s'\leq \xi\leq t'$, which shows the lemma.
\end{proof}

\begin{lemma}\label{lem2.5}
For all $t>0$ and $x>y>x'>y'$ we have
\begin{equation}\label{sec2-eq2.7}
|E\left[(u(t,x)-u(t,y))(u(t,x')-u(t,y'))\right]|
\leq \frac1{4\sqrt{t\pi}}(x-y)(x'-y')e^{-\frac{(y-x')^2}{4t}}.
\end{equation}
\end{lemma}
\begin{proof}
We have
\begin{align*}
E[&(u(t,x)-u(t,y))(u(t,x')-u(t,y'))]\\
&=R_{x,x'}(t,t)-R_{x,y'}(t,t)-R_{y,x'}(t,t)+R_{y,y'}(t,t)\\ &=\frac1{2\sqrt{\pi}}\left(\int_0^t\frac1{\sqrt{r}}
\exp\left\{-\frac{(x-x')^2}{4r}\right\}dr
-\int_0^t\frac1{\sqrt{r}}
\exp\left\{-\frac{(x-y')^2}{4r}\right\}dr\right.\\
&\qquad\qquad\left.-\int_0^t\frac1{\sqrt{r}}
\exp\left\{-\frac{(y-x')^2}{4r}\right\}dr
+\int_0^t\frac1{\sqrt{r}}
\exp\left\{-\frac{(y-y')^2}{4r}\right\}dr\right)\\
&=\frac{\sqrt{t}}{2\sqrt{\pi}}\left((x-x')\int_{(x-x')}^{ +\infty}\frac1{s^2}
e^{-\frac{s^2}{4t}}ds
-(x-y')\int_{(x-y')}^{+\infty}\frac1{s^2}
e^{-\frac{s^2}{4t}}ds\right.\\
&\qquad\qquad\left.-(y-x')\int_{(y-x')}^{+\infty}\frac1{s^2}
e^{-\frac{s^2}{4t}}ds+(y-y')\int_{(y-y')}^{+\infty}\frac1{s^2}
e^{-\frac{s^2}{4t}}ds\right)
\end{align*}
for all $t>0$ and $x>y>x'>y'$. Similar to the proof of Lemma~\ref{lem2.3} we define the function $f:{\mathbb R}_{+}\to {\mathbb R}_{+}$ by
$$
f(x)=x\int_x^{+\infty}\frac1{s^2}e^{-\frac{s^2}{4t}}ds.
$$
Then, we have
\begin{align*}
|E[&(u(t,x)-u(t,y))(u(t,x')-u(t,y'))]|\\
&=\frac{\sqrt{t}}{2\sqrt{\pi}}\left|f(x-x')-f(x-y')-f(y-x') +f(y-y')\right|\\
&=\frac{\sqrt{t}}{2\sqrt{\pi}}(x-y)\left|f'(\xi-x')-f'(\xi-y')\right|\\
&=\frac1{4\sqrt{t}\sqrt{\pi}}(x-y)\left|
\int_{\xi-x'}^{+\infty}e^{-\frac{s^2}{4t}}ds -\int_{\xi-y'}^{+\infty}e^{-\frac{s^2}{4t}}ds\right|\\
&=\frac1{2\sqrt{\pi}}(x-y)
\int_{\xi-x'}^{\xi-y'}\frac1{2\sqrt{t}}e^{-\frac{s^2}{4t}}ds\\
&\leq \frac1{4\sqrt{t\pi}}(x-y)(x'-y')e^{-\frac{(\xi-x')^2}{4t}}\\
&\leq \frac1{4\sqrt{t\pi}}(x-y)(x'-y')e^{-\frac{(y-x')^2}{4t}}
\end{align*}
for some $\xi\in [y,x]$ by Mean Value Theorem. This completes the proof.
\end{proof}
\begin{lemma}\label{lem2.6}
For all $t>s>0$ and $x\in {\mathbb R}$ denote
$\sigma^2_{t,x}=E\left[u(t,x)^2\right]$, $\sigma^2_{s,x}=E\left[u(s,x)^2\right]$,  $\mu_{t,s,x}=E\left[u(t,x)u(s,x)\right]$. Then we have
\begin{equation}\label{sec2-eq2.8}
\frac1{\pi}\sqrt{s(t-s)}\leq \sigma^2_{t,x}\sigma^2_{s,x}-\mu_{t,s,x}^2\leq \frac3{\pi}\sqrt{s(t-s)}.
\end{equation}
\end{lemma}
\begin{proof}
Given $t>s>0$ and $x\in {\mathbb R}$. We have
\begin{align*}
\sigma^2_{t,x}\sigma^2_{s,x}-\mu_{t,s,x}^2& =\frac{1}{2\pi}\left(2\sqrt{ts}-
\left((t+s)^{1/2}-(t-s)^{1/2}\right)^2\right)\\
&=\frac{1}{\pi}\left(\sqrt{ts}-t+\sqrt{t^2-s^2}\right)\\
&=\frac{t}{\pi}\left(\sqrt{z}+\sqrt{1-z^2}-1\right)
\end{align*}
with $z=\frac{s}{t}$. Clearly,
$$
\sqrt{z}+\sqrt{1-z^2}-1=\sqrt{z}-\left(1-\sqrt{1-z^2}\right)\geq \sqrt{z}-\sqrt{z^2}\geq \sqrt{z(1-z)}
$$
for all $0\leq z\leq 1$. Conversely, we have also that
\begin{align*}
0\leq \sqrt{z}+\sqrt{1-z}-1&=\sqrt{z(1-z)}+ \sqrt{z}+\left(\sqrt{1-z}-\sqrt{z(1-z)}\right)-1\\
&=\sqrt{z(1-z)}+ \sqrt{z}+\sqrt{1-z}\left(1-\sqrt{z}\right)-1\\
&\leq \sqrt{z(1-z)}+\left(\sqrt{z}-z\right)\leq 2\sqrt{z(1-z)},
\end{align*}
which gives
\begin{align*}
\sqrt{z}+\sqrt{1-z^2}-1&=\sqrt{z}+\sqrt{1-z+z-z^2}-1\\
&\leq \sqrt{z}+\sqrt{1-z}-1+\sqrt{z(1-z)}\leq 3\sqrt{z(1-z)}.
\end{align*}
This completes the proof.
\end{proof}
\begin{lemma}\label{lem2.7}
For all $t>0$ and $x>y$ denote
$\mu_{t,x,y}=E\left[u(t,x)u(t,y)\right]$. Then, under the conditions of Lemma~\ref{lem2.6} we have
\begin{equation}\label{sec2-eq2.9}
\sigma^2_{t,x}\sigma^2_{t,y}-\mu_{t,x,y}^2\asymp \frac{(x-y)t}{\sqrt{t}+x-y}.
\end{equation}
In particular, we have
\begin{equation}\label{sec2-eq2.10}
0\leq \sigma^2_{t,z}-\mu_{t,x,y}\asymp \frac{(x-y)\sqrt{t}}{\sqrt{t}+x-y}
\end{equation}
for all $t>0$ and $x,y,z\in {\mathbb R}$.
\end{lemma}
\begin{proof}
Given $t>0$ and $x>y$. We have
\begin{align*}
\sigma^2_{t,x}\sigma^2_{t,y}-\mu_{t,x,y}^2&
=\frac{t}{\pi}-\frac{1}{4\pi}
\left(\int_0^t\frac1{\sqrt{r}}
\exp\left\{-\frac{(x-y)^2}{4r}\right\}dr\right)^2\\
&=\frac{1}{4\pi}\left(4t-\left(\int_0^t\frac1{\sqrt{r}}
\exp\left\{-\frac{(x-y)^2}{4r}\right\}dr\right)^2\right)\\
&=\frac{1}{4\pi}\int_0^t\frac{dr}{\sqrt{r}}
\left(1-\exp\left\{-\frac{(x-y)^2}{4r}\right\}\right)\\
&\qquad\cdot
\int_0^t\frac{dr}{\sqrt{r}}
\left(1+\exp\left\{-\frac{(x-y)^2}{4r}\right\}\right)\\
&\asymp \sqrt{t}\int_0^t\frac{dr}{\sqrt{r}}
\left(1-\exp\left\{-\frac{(x-y)^2}{4r}\right\}\right)\\
&=2\sqrt{t}(x-y)\int_{\frac{x-y}{\sqrt{t}}}^\infty
\left(1-e^{-\frac{s^2}{4}}\right)\frac{ds}{s^2}.
\end{align*}
As in Lemma~\ref{lem2.3}, we define the function $f:{\mathbb R}_{+}\to {\mathbb R}_{+}$ by
$$
f(z)=z\int_z^{+\infty}\frac1{s^2} \left(1-e^{-\frac{s^2}{4}}\right)ds.
$$
Then $f$ is continuous in ${\mathbb R}_{+}$ and
$$
\lim_{z\to 0}\frac{f(z)}{z}=\int_0^{+\infty}\frac1{s^2} \left(1-e^{-\frac{s^2}{4}}\right)ds=\frac12\int_0^{+\infty} e^{-\frac{s^2}{4}}ds=\frac{\sqrt{\pi}}2
$$
and
$$
\lim_{z\to \infty}\frac{f(z)}{\frac{z}{1+z}}=1,
$$
which show that the functions $z\mapsto \frac{f(z)}{\frac{z}{1+z}}$ and $z\mapsto \frac{\frac{z}{1+z}}{f(z)}$
are bounded in $[0,\infty]$, i.e., there is
$$
\frac{x-y}{\sqrt{t}}\int_{\frac{x-y}{\sqrt{t}}}^\infty
\left(1-e^{-\frac{s^2}{4}}\right)\frac{ds}{s^2}\asymp \frac{\frac{x-y}{\sqrt{t}}}{1+\frac{x-y}{\sqrt{t}}} =\frac{x-y}{\sqrt{t}+x-y}.
$$
This shows that
\begin{align*}
\sigma^2_{t,x}\sigma^2_{t,y}-\mu_{t,x,y}^2
&\asymp \frac{(x-y)t}{\sqrt{t}+x-y},
\end{align*}
and the lemma follows.
\end{proof}

The second object in this section is to discuss Skorohod integrals associated with the solution process $u=\{u(t,x),t\geq 0,x\in {\mathbb R}\}$. From the above discussion we have known that the processes $B=\{B_t=u(t,\cdot),t\geq 0\}$ and $W=\{W_x=u(\cdot,x),x\in {\mathbb R}\}$ are neither semimartingales nor a Markov processes, so many of the powerful techniques from stochastic analysis are not available when dealing with the three processes. However, as a Gaussian process, one can develop the stochastic calculus of variations with respect to them. We refer to Al\'os {\em et al}~\cite{Nua1}, Nualart~\cite{Nua4} and the references therein for more details of stochastic calculus of Gaussian process.

Let ${\mathcal E}_t$ and ${\mathcal E}_x$ be respectively, the sets of linear combinations of elementary functions $\{1_{I_x},x\in \mathbb{R}\}$ and $\{1_{[0,t]},0\leq t\leq T\}$. Assume that $\mathcal{H}_t$ and ${\mathcal H}_{\ast}$ are the Hilbert spaces defined as the closure of ${\mathcal E}_t$ and ${\mathcal E}_x$ respect to the inner products
$$
\langle 1_{I_x},1_{I_y} \rangle_{{\mathcal H}_t}=\frac1{2\sqrt{\pi}}\int_0^t\frac1{\sqrt{s}}
\exp\left\{-\frac{(x-y)^2}{4s}\right\}ds=
\frac{|x-y|}{\sqrt{\pi}}\int_{\frac{|x-y|}{\sqrt{t}}}^\infty \frac1{s^2}e^{-\frac{s^2}{4}}ds
$$
and
$$
\langle 1_{[0,t]},1_{[0,s]} \rangle_{{\mathcal H}_\ast}=\frac1{\sqrt{2\pi}}
\left((t+s)^{1/2}-|t-s|^{1/2}\right),
$$
respectively. These maps $1_{I_x}\mapsto W_x$ and $1_{[0,t]}\mapsto B_t$ are two isometries between ${\mathcal E}_t$, ${\mathcal E}_x$ and the Gaussian spaces $W(\varphi)$, $B(\varphi)$ of $\{W_x,x\in {\mathbb R}\}$ and $\{B_t,t\geq 0\}$, respectively, which can be extended to $\mathcal{H}_t$ and ${\mathcal H}_{\ast}$, respectively. We denote these extensions by
$$
\varphi\mapsto W(\varphi)=\int_{\mathbb R}\varphi(y)dW_y
$$
and
$$
\psi\mapsto B(\psi)=\int_0^T\psi(s)dB_s,
$$
respectively.

Denote by ${\mathcal S}_t$ and ${\mathcal S}_{\ast}$ the sets of smooth functionals of the form
$$
F_t=f(W(\varphi_1),W(\varphi_2),\ldots, W(\varphi_n))
$$
and
$$
F_\ast=f(B(\psi_1),B(\psi_2),\ldots, B(\psi_n)),
$$
where $f\in C^{\infty}_b({\mathbb R}^n)$ ($f$ and all their
derivatives are bounded), $\varphi_i\in {\mathcal H}_t$ and $\psi_i\in {\mathcal H}_{\ast}$. The {\it derivative operators} $D^t$ and $D^{\ast}$ (the Malliavin derivatives) of functionals $F_t$ and $F_\ast$ of the above forms are defined as
$$
D^tF_t=\sum_{j=1}^n\frac{\partial f}{\partial
x_j}(W(\varphi_1),W(\varphi_2),
\ldots,W(\varphi_n))\varphi_j
$$
and
$$
D^{\ast}F_\ast=\sum_{j=1}^n\frac{\partial f}{\partial
x_j}(B(\psi_1),B(\psi_2),\ldots,B(\psi_n))\psi_j,
$$
respectively. These derivative operators $D^t,D^{\ast}$ are then closable from $L^2(\Omega)$ into $L^2(\Omega;{\mathcal H}_t)$ and $L^2(\Omega;{\mathcal H}_\ast)$, respectively. We denote by ${\mathbb D}^{t,1,2}$ and ${\mathbb D}^{\ast,1,2}$ the closures of ${\mathcal S}_t$, ${\mathcal S}_{\ast}$ and ${\mathcal S}$ with respect to the norm
$$
\|F_t\|_{t,1,2}:=\sqrt{E|F|^2+E\|D^{t}F_t\|^2_{{\mathcal H}_t}}
$$
and
$$
\|F_{\ast}\|_{\ast,1,2}:=\sqrt{E|F|^2+E\|D^{\ast}F_{\ast} \|^2_{{\mathcal H}_\ast}},
$$
respectively. The {\it divergence integrals} $\delta^{t}$ and $\delta^{\ast}$ are the adjoint of derivative operators $D^t$ and $D^{\ast}$, respectively. That are, we say that random variables $v\in L^2(\Omega;{\mathcal H}_t)$ $w\in L^2(\Omega;{\mathcal H}_\ast)$ belong to the domains of the
divergence operators $\delta^{t}$ and $\delta^{\ast}$, respectively, denoted by ${\rm {Dom}}(\delta^t)$ and ${\rm {Dom}}(\delta^\ast)$ if
$$
E\left|\langle D^{t}F_t,v\rangle_{{\mathcal H}_t}\right|\leq
c\|F_t\|_{L^2(\Omega)}, E\left|\langle D^{\ast}F_\ast,w\rangle_{{\mathcal H}_\ast}\right|\leq
c\|F_\ast\|_{L^2(\Omega)},
$$
respectively, for all $F_t\in {\mathcal S}_t$ and $F_\ast\in {\mathcal S}_t$. In these cases $\delta^{t}(v)$ and $\delta^{\ast}(w)$ are defined by the duality relationships
\begin{align}\label{sec2-2-eq2.1}
E\left[F_t\delta^t(v)\right]&=E\langle D^{t}F_t,v\rangle_{{\mathcal H}_t},\\  \label{sec2-2-eq2.2}
E\left[F_\ast\delta^\ast(w)\right]&=E\langle D^{\ast}F_\ast,w\rangle_{{\mathcal H}_\ast},
\end{align}
respectively, for any $v\in {\mathbb D}^{t,1,2}$ and $w\in {\mathbb D}^{\ast,1,2}$. We have ${\mathbb D}^{t,1,2}\subset {\rm {Dom}}(\delta^t)$ and ${\mathbb D}^{\ast,1,2}\subset {\rm {Dom}}(\delta^\ast)$. We will use the notations
$$
\delta^t(v)=\int_{\mathbb R}v_y\delta W_y,\qquad \delta^\ast(w)=\int_0^Tw_s\delta B_s
$$
to express the Skorohod integrals, and the indefinite Skorohod integrals is defined as
$$
\int_{I_x}v_y\delta W_y=\delta^t(v1_{I_x}),\quad \int_0^tw_s\delta B_s=\delta^\ast(w1_{[0,t]}),
$$
respectively. Denote
$$
\tilde{D}\in \{D^t,D^\ast\},\quad \tilde{\delta}=\{\delta^t,\delta^\ast\},\quad \tilde{{\mathbb D}}^{1,2}\in \{{\mathbb D}^{t,1,2},{\mathbb D}^{\ast,1,2}\}.
$$
We can localize the domains of the operators $\tilde{D}$ and $\tilde{\delta}$. If $\mathbb{L}$ is a class of random variables (or processes) we denote by $\mathbb{L}_{\rm loc}$ the set of random variables $F$ such that there exists a sequence $\{(\Omega_n, F^n), n\geq 1\}\subset {\mathscr F}\times \mathbb{L}$ with the following properties:

\hfill

(i) $\Omega_n\uparrow \Omega$, a.s.

(ii) $F=F^n$ a.s. on $\Omega_n$.

\hfill\\
If $F\in \tilde{\mathbb{D}}^{1,2}_{\rm loc}$, and $(\Omega_n, F^n)$ localizes $F$ in $\mathbb{D}^{1,2}$, then $\tilde{D}F$ is defined without ambiguity by $\tilde{D}F=\tilde{D}F^n$ on $\Omega_n$, $n\geq 1$. Then, if $v\in \tilde{\mathbb{D}}^{1,2}_{\rm loc}$, the divergence $\tilde{\delta}(v)$ is defined as a random variable determined by the conditions
$$
\tilde{\delta}(v)|_{\Omega_n}=\tilde{\delta}(v^n)|_{\Omega_n}\qquad {\rm { for\;\; all\;\;}} n\geq 1,
$$
where $(\Omega_n, v^n)$ is a localizing sequence for $v$, but it may depend on the localizing sequence.

\section{The quadratic covariation of process $\{u(\cdot,x),x\in {\mathbb R}\}$}\label{sec4}
In this section, we study the existence of the
PQC $[f(u(t,\cdot)),u(t,\cdot)]^{(SQ)}$. Recall that
$$
I_\varepsilon^2(f,x,t)=\frac1{\varepsilon} \int_{I_x}\left\{f(u(s,y+\varepsilon))-f(u(s,y))\right\}
(u(s,y+\varepsilon)-u(s,y))dy
$$
for $\varepsilon>0$ and $x\in {\mathbb R}$, and
\begin{equation}\label{sec4-eq4.1}
[f(u(t,\cdot)),u(t,\cdot)]^{(SQ)}_x=\lim_{\varepsilon\downarrow
0}I_\varepsilon^2(f,x,t),
\end{equation}
provided the limit exists in probability. In this section we fix a time parameter $t>0$ and recall that
$$
W=\{W_x=u(\cdot,x),x\in {\mathbb R}\}.
$$

Recall that the local H\"{o}der index $\gamma_0$
of a continuous paths process $\{X_t: t\geq 0\}$ is the supremum of the exponents $\gamma$ verifying, for any $T>0$:
$$
P(\{\omega: \exists L(\omega)>0, \forall s,t \in[0,T],
|X_t(\omega)-X_s(\omega)|\leq L(\omega)|t-s|^\gamma\})=1.
$$
Recently, Gradinaru-Nourdin~\cite{Grad3} introduced the following
very useful result:
\begin{lemma}\label{Grad-Nourdin}
Let $g:{\mathbb R}\to {\mathbb R}$ be a function satisfying
\begin{equation}\label{eq4.2-Gradinaru--Nourdin}
|g(x)-g(y)|\leq C|x-y|^a(1+x^2+y^2)^b,\quad (C>0,0<a\leq 1,b>0),
\end{equation}
for all $x,y\in {\mathbb R}$ and let $X$ be a locally H\"older
continuous paths process with index $\gamma\in (0,1)$. Assume that
$V$ is a bounded variation continuous paths process. Set
$$
X^{g}_\varepsilon(t)=\int_0^tg\left(\frac{X_{s+\varepsilon}-X_s
}{\varepsilon^\gamma}\right)ds
$$
for $t\geq 0$, $\varepsilon>0$. If for each $t\geq 0$, as
$\varepsilon\to 0$,
\begin{equation}\label{condition}
\|X^{g}_\varepsilon(t)-V_t\|_{L^2}^2=O(\varepsilon^\alpha)
\end{equation}
with $\alpha>0$, then, $\lim_{\varepsilon\to
0}X^{g}_\varepsilon(t)=V_t$ almost surely, for any $t\geq 0$, and if $g$ is non-negative, for any continuous stochastic process $\{Y_t:\;t\geq
0\}$,
\begin{equation}
\lim_{\varepsilon\to 0}
\int_0^tY_sg\left(\frac{X_{s+\varepsilon}-X_s}{\varepsilon^\gamma} \right)ds
\longrightarrow \int_0^tY_sdV_s,
\end{equation}
almost surely, uniformly in $t$ on each compact interval.
\end{lemma}
According to the lemma above we get the next proposition.
\begin{proposition}\label{prop4.1}
Let $f\in C^1({\mathbb R})$. We have
\begin{equation}\label{sec4-eq3.5}
[f(W),W]^{(SQ)}_x=\int_{I_x}f'(W_y)dy
\end{equation}
and in particular, we have
$$
[W,W]^{(SQ)}_x=|x|
$$
for all $x\in {\mathbb R}$.
\end{proposition}
\begin{proof}
In fact, the H\"older continuity of the process $W=\{W_x=u(\cdot,x),x\in {\mathbb R}\}$ yields
\begin{align*}
\lim_{\varepsilon\downarrow 0}\frac{1}{\varepsilon}
\int_{I_x}o((W_{y+\varepsilon}-W_y)) (W_{y+\varepsilon}-W_y)^2dy=0
\end{align*}
for all $x\in {\mathbb R}$, almost surely. It follows that
\begin{align*}
\lim_{\varepsilon\downarrow 0}&\frac1{\varepsilon}
\int_{I_x}\left\{f(W_{y+\varepsilon})-f(W_y)\right\}
(W_{y+\varepsilon}-W_y)dy\\
&=\lim_{\varepsilon\downarrow 0}\frac{1}{\varepsilon}
\int_{I_x}
f'(W_y)(W_{y+\varepsilon}-W_y)^2dy
\end{align*}
almost surely. Thus, to end the proof we need to prove the next convergence
$$
[W,W]^{(SQ)}_x =|x|
$$
almost surely. That is, for each $t\geq 0$
\begin{equation}
\left\|W^\varepsilon(x)-|x|
\right\|_{L^2}^2 =O(\varepsilon^\alpha)
\end{equation}
with some $\alpha>0$, as $\varepsilon\to 0$, by the above lemma, where
$$
W^\varepsilon(x)=\frac1{\varepsilon} \int_{I_x}(W_{y+\varepsilon}-W_y)^2dy.
$$
We have
$$
E\left|W^\varepsilon(x)-|x| \right|^2=\frac{1}{\varepsilon^2}\int_{I_x}\int_{I_x}
B_\varepsilon(y,z)dydz
$$
for $x\in {\mathbb R}$ and $\varepsilon>0$, where
\begin{align*}
B_\varepsilon(y,z):&=E\left(
(W_{y+\varepsilon}-W_y)^2
-\varepsilon\right)\left(
(W_{z+\varepsilon}-W_z)^2-\varepsilon \right)\\
&=E(W_{y+\varepsilon}-W_y)^2(W_{z+\varepsilon}-W_z)^2
+\varepsilon^2\\
&\qquad-\varepsilon
E\left((W_{y+\varepsilon}-W_y)^2+(W_{z+\varepsilon}-W_z)^2\right).
\end{align*}
Recall that
\begin{align*}
E[(W_{y+\varepsilon}&-W_y)^2]= \frac1{\sqrt{\pi}}\int_0^t\frac{1}{\sqrt{t-r}}\left(
1-\exp\left\{-\frac{\varepsilon^2}{4(t-r)}\right\}\right)dr\\
&=\frac1{\sqrt{\pi}}\int_0^t\frac{1}{\sqrt{r}}\left(
1-e^{-\frac{\varepsilon^2}{4r}}\right)dr\\
&=\sqrt{\frac2{\pi}}\varepsilon \int_{\frac{\varepsilon}{\sqrt{2t}}}^\infty\frac{1}{s^2}\left(
1-e^{-\frac{s^2}{2}}\right)ds \equiv \phi_{t,y}(\varepsilon)+\varepsilon,
\end{align*}
where
\begin{align*}
\phi_{t,y}(\varepsilon)&=\sqrt{\frac2{\pi}}\varepsilon \int_{\frac{\varepsilon}{\sqrt{2t}}}^\infty\frac{1}{s^2}\left(
1-e^{-\frac{s^2}{2}}\right)ds-\varepsilon.
\end{align*}
Noting that
\begin{align*}
E[(W_{y+\varepsilon}&-W_y)^2(W_{z+\varepsilon}-W_z)^2]\\
&=
E\left[(W_{y+\varepsilon}-W_y)^2\right]
E\left[(W_{z+\varepsilon}-W_z)^2\right]\\
&\hspace{1cm}+2\left(E\left[(W_{y+\varepsilon}-W_y)
(W_{z+\varepsilon}-W_z)\right]\right)^{2}
\end{align*}
for all $\varepsilon>0$ and $y,z\in I_x$, we get
\begin{align*}
B_\varepsilon(y,z)&=\phi_{t,y}(\varepsilon)\phi_{t,y}(\varepsilon) +2(\mu_{y,z})^2
\end{align*}
where $\mu_{y,z}:=E\left[(W_{y+\varepsilon}-W_y)
(W_{z+\varepsilon}-W_z)\right]$.

Now, let us estimate the above function $\varepsilon\mapsto \phi_{t,y}(\varepsilon)$. We have
\begin{align*}
\phi_{t,y}(\varepsilon)
&=\sqrt{\frac2{\pi}}\varepsilon\left(\int_{\frac{\varepsilon}{ \sqrt{2t}}}^\infty\frac{1}{s^2}\left(
1-e^{-\frac{s^2}{2}}\right)ds-\sqrt{\frac{\pi}2}\right)\\
&=\sqrt{\frac2{\pi}}\varepsilon\left(\int_{\frac{\varepsilon}{ \sqrt{2t}}}^\infty\frac{1}{s^2}\left(
1-e^{-\frac{s^2}{2}}\right)ds-\int_0^\infty \frac{1}{s^2}\left(
1-e^{-\frac{s^2}{2}}\right)ds\right)\\
&=-\sqrt{\frac2{\pi}}\varepsilon
\int_0^{\frac{\varepsilon}{\sqrt{2t}}}\frac{1}{s^2}\left(
1-e^{-\frac{s^2}{2}}\right)ds\sim \frac1{2\sqrt{t\pi}}\varepsilon^2\quad (\varepsilon\to 0)
\end{align*}
by the fact
$$
\frac1{\sqrt{2\pi}} \int_0^\infty \frac1{s^2}\left(1-e^{-\frac{s^2}{2}}\right)ds=\frac1{\sqrt{2\pi}} \int_0^\infty e^{-\frac{s^2}{2}}ds=\frac12,
$$
which gives
\begin{align*}
\frac{1}{\varepsilon^2}&\int_{I_x}\int_{I_x} |\phi_{t,y}(\varepsilon)\phi_{t,z}(\varepsilon)|dydz\sim
\frac1{4t\pi}\varepsilon^2x^2\quad (\varepsilon\to 0).
\end{align*}
It follows from Lemma~\ref{lem2.4} that there is a
constant $\alpha>0$ such that
\begin{align*}
\lim_{\varepsilon\downarrow 0}\frac{1}{
\varepsilon^{1+\alpha}}\int_{I_x}\int_{I_x}B_\varepsilon(y,z)dydz=0
\end{align*}
for all $t>0$ and $x\in {\mathbb R}$, which gives the desired estimate
$$
\left\|W^\varepsilon(x)-x\right\|_{L^2}^2=O\left(
\varepsilon^\alpha\right)\qquad (\varepsilon\to 0)
$$
for all $x\in {\mathbb R}$ and some $\alpha>0$.

Notice that $g(y)=y^2$ satisfies the
condition~\eqref{eq4.2-Gradinaru--Nourdin}. We obtain the proposition by taking $Y_y=f'(W_y)$ for $y\in {\mathbb R}$.
\end{proof}

Now, we discuss the existence of the PQC $[f(W),W]^{(SQ)}$. Consider the decomposition
\begin{equation}\label{sec4-eq4.000000}
\begin{split}
I_\varepsilon^1(f,x,t)&=\frac1{\varepsilon}\int_{I_x} f(W_{y+\varepsilon})(W_{y+\varepsilon}-W_y)dy\\
&\hspace{2cm}-\frac1{\varepsilon}\int_{I_x} f(W_y)(W_{y+\varepsilon}-W_y)dy\\
&\equiv I_\varepsilon^{1,+}(f,x,t)-I_\varepsilon^{1,-}(f,x,t)
\end{split}
\end{equation}
for $\varepsilon>0$, and define the set
$$
{\mathscr H}_t=\{f\,:\,{\text { Borel functions on ${\mathbb R}$ such that $\|f\|_{{\mathscr H}_t}<\infty$}}\},
$$
where
\begin{align*}
\|f\|_{{\mathscr H}_t}^2:&=\frac{|x|}{\sqrt[4]{4\pi t}}\int_{\mathbb R}|f(z)|^2\left(\sqrt{t}+z^2\right) e^{-\frac{\sqrt{\pi}z^2}{2\sqrt{t}}}dz.
\end{align*}
Then ${\mathscr H}_t=L^2({\mathbb R},\mu(dz))$ with
$$
\mu(dz)=\left(\frac{|x|}{\sqrt[4]{4\pi t}}\left(\sqrt{t}+z^2\right) e^{-\frac{\sqrt{\pi}z^2}{2\sqrt{t}}}\right)dz
$$
and $\mu({\mathbb R})=C|x|<\infty$, which implies that the set ${\mathscr E}$ of elementary functions of the form
$$
f_\triangle(z)=\sum_{i}f_{i}1_{(x_{i-1},x_{i}]}(z)
$$
is dense in ${\mathscr H}_t$, where $f_i\in {\mathbb R}$ and $\{x_i,0\leq i\leq l\}$ is a finite sequence of real numbers such that $x_i<x_{i+1}$. Moreover, ${\mathscr H}_t$ includes all Borel functions $f$ satisfying the condition
\begin{equation}
|f(z)|\leq Ce^{\beta {z^2}},\quad z\in {\mathbb R}
\end{equation}
with $0\leq \beta<\frac{\sqrt{\pi}}{4\sqrt{t}}$.
\begin{theorem}\label{th3.1}
Let $f\in {\mathscr H}_t$. Then, the PQC $[f(W), W]^{(SQ)}$ exists in $L^2(\Omega)$ and
\begin{align}
E\left|[f(W), W]^{(SQ)}_x\right|^2\leq C \|f\|_{{\mathscr H}_t}^2
\end{align}
for all $x\in {\mathbb R}$.
\end{theorem}

In order to prove the theorem we claim that the following two statements with $f\in {\mathscr H}_t$:
\begin{itemize}
\item [(1)] for any $\varepsilon>0$ and $x\in {\mathbb R}$, $I_\varepsilon^{1,\pm}(f,x,\cdot)\in L^2(\Omega)$. That is,
\begin{align*}
&E\left|I_\varepsilon^{1,-}(f,x,\cdot)\right|^2\leq C \|f\|_{{\mathscr H}_t}^2,\\
&E\left|I_\varepsilon^{1,+}(f,x,\cdot)\right|^2\leq C \|f\|_{{\mathscr H}_t}^2.
\end{align*}
\item [(2)] $I_\varepsilon^{1,-}(f,x,t)$ and $I_\varepsilon^{1,+}(f,x,t)$ are two Cauchy's sequences in $L^2(\Omega)$ for all $t>0$ and $x\in {\mathbb R}$. That is,
\begin{equation*}
E\left|I_{\varepsilon_1}^{1,-}(f,x,t)-I_{\varepsilon_2}^{1,-}(f,x,t) \right|^2\longrightarrow 0,
\end{equation*}
and
\begin{equation*}
E\left|I_{\varepsilon_1}^{1,+}(f,x,t) -I_{\varepsilon_2}^{1,+}(f,x,t)\right|^2
\longrightarrow 0
\end{equation*}
for all $x\in {\mathbb R}$, as $\varepsilon_1,\varepsilon_2\downarrow 0$.
\end{itemize}

We split the proof of two statements into two parts.

\begin{proof}[Proof of the statement (1)]
Recall that $W_x:=u(\cdot,x)$. We have
\begin{align*}
E|I_\varepsilon^{1,-}(f,x,\cdot)|^2&=\frac{1}{\varepsilon^2}
\int_{I_x}\int_{I_x}dydy'E\left[f({W_y})f({W_{y'}}) (W_{y+\varepsilon}-{W_y}) (W_{y'+\varepsilon}-{W_{y'}})\right]
\end{align*}
for all $\varepsilon>0$ and $x\in {\mathbb R}$. Now, let us estimate the expression
$$
\Phi_{\varepsilon_1,\varepsilon_2}(y,y'):
=E\left[f({W_y})f({W_{y'}})(W_{y+\varepsilon_1}-{W_y}) (W_{y'+\varepsilon_2}-{W_{y'}})\right]
$$
for all $\varepsilon_1,\varepsilon_2>0$ and $y,y'\in {\mathbb R}$. To estimate the above expression, it is enough to assume that $f\in {\mathscr E}$ by denseness, and moreover, by approximating we can assume that $f$ is an infinitely differentiable function with compact support. It follows from the duality relationship~\eqref{sec2-2-eq2.1} that
\begin{equation}\label{sec7-eq7.9}
\begin{split}
\Phi_{\varepsilon_1,\varepsilon_2}(y,y') &=E\left[f({W_y})f({W_{y'}})(W_{y+\varepsilon_1} -{W_y})\int_{y'}^{y'+\varepsilon_2}\delta W(l)\right]\\
&=E\left[{W_y}(W_{y'+\varepsilon_2}-{W_{y'}})\right] E\left[f'({W_y})f({W_{y'}})(W_{y+\varepsilon_1}-{W_y})\right]\\
&\quad+E\left[{W_{y'}}(W_{y'+\varepsilon_2}-{W_{y'}})\right] E\left[f({W_y})f'({W_{y'}})(W_{y+\varepsilon_1}-{W_y})\right]\\
&\quad+E\left[(W_{y+\varepsilon_1}-{W_y}) (W_{y'+\varepsilon_2}-{W_{y'}})\right]
E\left[f({W_y})f({W_{y'}}\right]\\
&=E\left[{W_y}(W_{y'+\varepsilon_2}-{W_{y'}})\right] E\left[{W_y}(W_{y+\varepsilon_2}-{W_y})\right] E\left[f''({W_y})f({W_{y'}})\right]\\
&\quad+E\left[{W_y}(W_{y'+\varepsilon_2}-{W_{y'}})\right] E\left[{W_{y'}}(W_{y+\varepsilon_2}-{W_y})\right]
E\left[f'({W_y})f'({W_{y'}})\right]\\
&\quad+E\left[{W_{y'}}(W_{y'+\varepsilon_2}-{W_{y'}})\right] E\left[{W_y}(W_{y+\varepsilon_1}-{W_y})\right]
E\left[f'({W_y})f'({W_{y'}}))\right]\\
&\quad+E\left[{W_{y'}}(W_{y'+\varepsilon_2}-{W_{y'}})\right] E\left[{W_{y'}}(W_{y+\varepsilon_1}-{W_y})\right]
E\left[f({W_y})f''({W_{y'}})\right]\\
&\quad+E\left[(W_{y+\varepsilon_1}-{W_y}) (W_{y'+\varepsilon_2}-{W_{y'}})\right]
E\left[f({W_y})f({W_{y'}}\right]\\
&\equiv \sum_{j=1}^5\Psi_j(y,y',\varepsilon_1,\varepsilon_2)
\end{split}
\end{equation}
for all $y,y'\in {\mathbb R}$ and $\varepsilon_1,\varepsilon_2>0$. In order to end the proof we claim to estimate
$$
\Lambda_j:=\frac{1}{\varepsilon^2}\int_{I_x}\int_{I_x} \Psi_j(y,y',\varepsilon,\varepsilon)dydy',\quad j=1,2,3,4,5
$$
for all $\varepsilon>0$ small enough.

For $j=5$, from the fact
\begin{align*}
|E&\left[(W_{y+\varepsilon}-{W_y}) (W_{y'+\varepsilon}-{W_{y'}})\right]|\leq \varepsilon
\end{align*}
for $0<|y-y'|\leq \varepsilon$, we have
\begin{align*}
\frac{1}{\varepsilon^2}\int_{\substack{|y-y'|\leq \varepsilon\\ y,y'\in I_x}}&|\Psi_5(y,y',\varepsilon,\varepsilon)|dydy' \leq \frac{1}{\varepsilon}\int_{\substack{|y-y'|\leq \varepsilon\\ y,y'\in I_x}}
E\left|f({u_y})f({W_{y'}})\right|dydy'\\
&\leq \frac{1}{2\varepsilon}\int_{\substack{|y-y'|\leq \varepsilon\\ y,y'\in I_x}}
E\left[f({W_y})|^2+|f({W_{y'}})|^2\right]dydy'\\
&\leq \frac{1}{\varepsilon}\int_{\substack{|y-y'|\leq \varepsilon\\ y,y'\in I_x}}E\left|f({W_y})\right|^2dydy'\\
&\leq \int_{I_x}E\left|f({W_y})\right|^2dy=\|f\|_{{\mathscr H}_t}^2
\end{align*}
for all $\varepsilon>0$ and $x\in {\mathbb R}$. Moreover, for $|y-y'|>\varepsilon$ we have
\begin{align*}
|E&\left[(W_{y+\varepsilon}-{W_y}) (W_{y'+\varepsilon}-{W_{y'}})\right]|\leq \frac1{4\sqrt{t\pi}}\varepsilon^2 e^{-\frac{|y-y'-\varepsilon|^2}{4t}}
\end{align*}
by~\eqref{sec2-eq2.7}, which deduces
\begin{align*}
\frac{1}{\varepsilon^2}&\int_{\substack{|y-y'|>\varepsilon\\ y,y'\in I_x}}|\Psi_5(y,y',\varepsilon,\varepsilon)|dydy'\\
&\leq
\frac1{4\sqrt{t\pi}}\int_{\substack{|y-y'|>\varepsilon\\ y,y'\in I_x}}E\left|f({W_y})f({W_{y'}})\right|
e^{-\frac{|y-y'-\varepsilon|^2}{4t}}dydy'\\
&\leq \frac1{8\sqrt{t\pi}}\int_{\substack{|y-y'|>\varepsilon\\ y,y'\in I_x}}
E\left[f({W_y})|^2+|f({W_{y'}})|^2\right] e^{-\frac{|y-y'-\varepsilon|^2}{4t}}dydy'\\
&\leq \frac1{4\sqrt{t\pi}}\int_{\substack{|y-y'|>\varepsilon\\ y,y'\in I_x}}E\left|f({W_y})\right|^2 e^{-\frac{|y-y'-\varepsilon|^2}{4t}}dydy'\\
&\leq \frac1{2}\int_{I_x}
E\left|f({W_y})\right|^2dy\int_{-\infty}^\infty
\frac1{\sqrt{2\pi(2t)}}
e^{-\frac{|y-y'-\varepsilon|^2}{4t}}dy'\\
&=\frac1{2}\int_{I_x}E\left|f({W_y})\right|^2dy =\frac12\|f\|_{{\mathscr H}_t}^2
\end{align*}
for all $\varepsilon>0$ and $x\in {\mathbb R}$. This shows that
$$
\Lambda_5=\frac{1}{\varepsilon^2}\left|\int_{I_x}\int_{I_x} \Psi_5(y,y',\varepsilon,\varepsilon)dydy'\right|\leq \|f\|_{{\mathscr H}_t}^2
$$
for all $\varepsilon>0$ and $x\in {\mathbb R}$.

Next, let us estimate $\sum\limits_{j=1}^4\Lambda_j$. We have
\begin{align*}
E\left[f''({W_y})f({W_{y'}})\right]&=\int_{\mathbb{R}^2}
f(x)f(x')\frac{\partial^{2}}{\partial x^2}
\varphi(x,x')dxdx'\\
&=\int_{\mathbb{R}^2} f(x)f(x')\left\{\frac1{\rho^4}(\sigma^2_{t,y'}x-\mu_{t,y,y'}
x')^2-\frac{\sigma^2_{t,y'}}{\rho^2}\right\}\varphi(x,x')dxdx'
\end{align*}
and
\begin{align*}
E[&f'({W_y})f'({W_{y'}})]=\int_{\mathbb{R}^2}
f(x)f(x')\frac{\partial^{2}}{\partial x\partial
x'}\varphi(x,x')dxdx'\\
&=\int_{\mathbb{R}^2} f(x)f(x')\left\{\frac1{\rho^4}(\sigma^2_{t,y}x'-\mu_{t,y,y'}
x)(\sigma^2_{t,y'}x-\mu_{t,y,y'} x')+\frac{\mu_{t,y,y'}}{\rho^2}\right\}\varphi(x,x')dxdx',
\end{align*}
where $\rho^2=\sigma^2_{t,y}\sigma^2_{t,y'}-\mu_{t,y,y'}^2$ and
$\varphi(x,y)$ is the density function of $({W_y},{W_{y'}})$. That is
$$
\varphi(x,x')=\frac1{2\pi\rho}\exp\left\{-\frac{1}{2\rho^2}\left(
\sigma^2_{t,y'}x^2-2\mu_{t,y,y'}xx'+\sigma^2_{t,y}{x'}^2\right) \right\}.
$$
Combining this with the identity
\begin{align*}
(\sigma^2_{t,y}x'-\mu_{t,y,y'}
x)&(\sigma^2_{t,y'}x-\mu_{t,y,y'} x')\\
&=\rho^2x'\left(x-\frac{\mu_{t,y,y'}}{\sigma^2_{t,y'}}x'\right) -\mu_{t,y,y'}\sigma^2_{t,y'}\left(x -\frac{\mu_{t,y,y'}}{\sigma^2_{t,y'}}x'\right)^2,
\end{align*}
we get
\begin{align*}
E[f''({W_y})&f({W_{y'}})]+E[f'({W_y})f'({W_{y'}})]\\
&=\frac{\mu_{t,y,y'}-\sigma^2_{t,y'}}{\rho^2}\int_{\mathbb{R}^2} f(x)f(x') \varphi(x,x')dxdx'\\
&\qquad+\frac1{\rho^2}\int_{\mathbb{R}^2} f(x)f(x')
x'\left(x-\frac{\mu_{t,y,y'}}{\sigma^2_{t,y'}}x'\right) \varphi(x,x')dxdx'\\
&\qquad+\frac{\sigma^2_{t,y'}}{\rho^4}\left(\sigma^2_{t,y'}- \mu_{t,y,y'}\right)\int_{\mathbb{R}^2} f(x)f(x')
\left(x-\frac{\mu_{t,y,y'}}{\sigma^2_{t,y'}}x'\right)^2 \varphi(x,x')dxdx'\\
&\equiv \Upsilon_1+\Upsilon_2+\Upsilon_3.
\end{align*}
A straightforward calculation shows that
\begin{align*}
\int_{\mathbb{R}^2}|f(x')|^2 &\left(x-\frac{\mu_{t,y,y'}}{\sigma^2_{t,y'}}x'\right)^{2m} \varphi(x,x')dxdx'\\
&=C_m\left(\frac{\rho^2}{\sigma^2_{t,y'}}\right)^m \int_{\mathbb{R}}|f(x')|^2\frac1{\sqrt{2\pi}\sigma_{t,y'}} e^{-\frac{{x'}^2}{2\sigma^2_{t,y'}}}dx'\\
&\leq C_m\left(\frac{\rho^2}{\sqrt{t}}\right)^m \int_{\mathbb{R}}|f(x)|^2\frac1{\sqrt[4]{t}} e^{-\frac{\sqrt{\pi}x^2}{2\sqrt{t}}}dx
\end{align*}
for all $m\geq 1$ and
\begin{align*}
\int_{\mathbb{R}^2}&|f(x)x'|^2 \varphi(x,x')dxdx'\\
&=\int_{\mathbb{R}}|f(x)|^2\frac1{\sqrt{2\pi}\sigma_{t,y}} e^{-\frac{{x}^2}{2\sigma^2_{t,y}}}dx
\int_{\mathbb{R}}|x'|^2\frac{\sigma_{t,y}}{\sqrt{2\pi}\rho} e^{-\frac{\sigma^2_{t,y}}{2\rho^2}
\left(x'-\frac{\mu_{t,y,y'}}{\sigma^2_{t,y}}x\right)^2}dx'\\
&=\int_{\mathbb{R}}|f(x)|^2\frac1{\sqrt{2\pi}\sigma_{t,y}} e^{-\frac{{x}^2}{2\sigma^2_{t,y}}}dx
\left(\frac{\rho^2}{\sigma^2_{t,y}}+
\frac{\mu_{t,y,y'}^2}{\sigma^4_{t,y}}x^2\right)\\
&\leq C\frac1{\sqrt[4]{\pi t}}\int_{\mathbb{R}}|f(x)|^2 e^{-\frac{\sqrt{\pi}x^2}{2\sqrt{t}}}\left(\sqrt{t}+x^2\right)dx
\end{align*}
since $\sigma^2_{t,y}=\sigma^2_{t,y'}=\sqrt{\frac{t}{\pi}}$. It follows that
\begin{align*}
|\Upsilon_1|&\leq \frac1{\rho^2}\left(\int_{\mathbb{R}^2}|f(x)x'|^2
\varphi(x,x')dxdx'\int_{\mathbb{R}^2}|f(x')|^2
\left|x-\frac{\mu_{t,y,y'}}{\sigma^2_{t,y'}}x'\right|^2 \varphi(x,x')dxdx'\right)^{\frac12}\\
&\leq C\frac{1}{\rho\sqrt[4]{t}}\int_{\mathbb{R}}|f(x)|^2 \frac1{\sqrt[4]{t}} e^{-\frac{\sqrt{\pi}x^2}{2\sqrt{t}}}\left(\sqrt{t}+x^2\right)dx
\end{align*}
and
\begin{align*}
|\Upsilon_3|&\leq
\frac{\sigma^2_{t,y'}}{\rho^4}\left|\sigma^2_{t,y'}- \mu_{t,y,y'}\right|\\
&\quad\cdot\left( \int_{\mathbb{R}^2}|f(x)|^2\varphi(x,x')dxdx'
\int_{\mathbb{R}^2}
|f(x')|^2\left(x-\frac{\mu_{t,y,y'}}{\sigma^2_{t,y'}}x'\right)^4 \varphi(x,x')dxdx'\right)^{1/2}\\
&\leq \frac{C\left|\sigma^2_{t,y'}-\mu_{t,y,y'}\right|}{\rho^2}
\int_{\mathbb{R}}|f(x)|^2\frac1{\sqrt[4]{t}} e^{-\frac{\sqrt{\pi}x^2}{2\sqrt{t}}}dx.
\end{align*}
Thus, we get the estimate
\begin{equation}\label{sec3-eq3.999}
\begin{split}
|E[f''({W_y})f({W_{y'}})]&+E[f'({W_y})f'({W_{y'}})]|\leq |\Upsilon_1|+|\Upsilon_2|+|\Upsilon_3|\\
&\leq \int_{\mathbb{R}}|f(x)|^2 \frac1{\sqrt[4]{t}} e^{-\frac{\sqrt{\pi}x^2}{2\sqrt{t}}}\left(\sqrt{t}+x^2\right)dx
\end{split}
\end{equation}
and
\begin{equation}\label{sec3-eq3.1000}
|E\left[f''({W_y})f({W_{y'}})\right]|\leq \frac{C}{|y-y'|}\int_{\mathbb{R}}|f(x)|^2 \frac1{\sqrt[4]{t}} e^{-\frac{\sqrt{\pi}x^2}{2\sqrt{t}}}dx
\end{equation}
by Lemma~\ref{lem2.7} and Lemma~\ref{lem2.3}. Now, we can estimate $\sum\limits_{j=1}^4\Lambda_j$. We have
\begin{align*}
\sum_{j=1}^4&\Psi_j(y,y',\varepsilon,\varepsilon)\\
&=E\left[{W_y}(W_{y'+\varepsilon}-{W_{y'}})\right] E\left[({W_y}-{W_{y'}})(W_{y+\varepsilon}-{W_y})\right] E\left[f''({W_y})f({W_{y'}})\right]\\
&\;\;+E\left[{W_y}(W_{y'+\varepsilon}-{W_{y'}})\right] E\left[{W_{y'}}(W_{y+\varepsilon}-{W_y})\right]\\ &\qquad\qquad\qquad\cdot\left(E\left[f'({W_y})f'({W_{y'}})\right] +E\left[f''({W_y})f({W_{y'}})\right]\right)\\
&\;\;+E\left[{W_{y'}}(W_{y'+\varepsilon}-{W_{y'}})\right] E\left[{W_y}(W_{y+\varepsilon}-{W_y})\right]\\
&\qquad\qquad\qquad\cdot\left(E\left[f'({W_y})f'({W_{y'}})\right]+
E\left[f({W_y})f''({W_{y'}})\right]\right)\\
&\;\;+E\left[{W_{y'}}(W_{y'+\varepsilon}-{W_{y'}})\right] E\left[({W_{y'}}-{W_y})(W_{y+\varepsilon}-{W_y})\right]
E\left[f({W_y})f''({W_{y'}})\right].
\end{align*}
Combining this with~\eqref{sec3-eq3.999},~\eqref{sec3-eq3.1000},  Lemma~\ref{lem2.3} and Lemma~\ref{lem2.5}, we get
\begin{align*}
\left|\sum\limits_{j=1}^4\Lambda_j\right| \leq \frac{1}{\varepsilon^2}\int_{I_x}\int_{I_x} \left|\sum\limits_{j=1}^4\Psi_j(y,y',\varepsilon,\varepsilon)
\right|dydy'\leq C\|f\|_{{\mathscr H}_t}^2
\end{align*}
for all $\varepsilon>0$ and $x\in {\mathbb R}$. This shows that
\begin{align*}
E\left|I_\varepsilon^{1,-}(f,x,\cdot)\right|^2\leq C \|f\|_{{\mathscr H}_t}^2.
\end{align*}
Similarly, one can shows that the estimate
\begin{align*}
E\left|I_\varepsilon^{1,+}(f,x,\cdot)\right|^2\leq C \|f\|_{{\mathscr H}_t}^2,
\end{align*}
and the first statement follows.
\end{proof}

\begin{proof}[Proof of the statement (2)]
Without loss of generality we assume that $\varepsilon_1>\varepsilon_2$. We prove only the first convergence and similarly one can prove the second convergence. We have
\begin{align*}
E\bigl|&I_{\varepsilon_1}^{1,-}(f,x,t) -I_{\varepsilon_2}^{1,-}(f,x,t) \bigr|^2\\
&=\frac1{\varepsilon_1^2}\int_{I_x}\int_{I_x}Ef(W_y)f(W_{y'})
(W_{y+\varepsilon_1}-W_{y})(W_{y'+\varepsilon_1}-W_{y'})dydy'\\
&\qquad-2
\frac1{\varepsilon_1\varepsilon_2}\int_{I_x}\int_{I_x} Ef(W_y)f(W_{y'})
(W_{y+\varepsilon_1}-W_{y})(W_{y'+\varepsilon_2}-W_{y'})dydy'\\
&\qquad+\frac1{\varepsilon_2^2}\int_{I_x}\int_{I_x}Ef(W_y)f(W_{y'})
(W_{y+\varepsilon_2}-W_{y})(W_{y'+\varepsilon_2}-W_{y'})dydy'\\
&\equiv \frac1{\varepsilon_1^2\varepsilon_2}\int_{I_x}\int_{I_x}
\left\{\varepsilon_2\Phi_{y,y'}(1,\varepsilon_1)-\varepsilon_1
\Phi_{y,y'}(2,\varepsilon_1,\varepsilon_2)\right\}dydy'\\
&\qquad+
\frac1{\varepsilon_1\varepsilon_2^2}\int_{I_x}\int_{I_x}\left\{
\varepsilon_1\Phi_{y,y'}(1,\varepsilon_2)-\varepsilon_2
\Phi_{y,y'}(2,\varepsilon_1,\varepsilon_2)\right\}dydy',
\end{align*}
for all $\varepsilon_1,\varepsilon_2>0$ and $x\in {\mathbb R}$,
where
$$
\Phi_{y,y'}(1,\varepsilon) =E\left[f(W_y)f(W_{y'})(W_{y+\varepsilon}
-W_y)(W_{y'+\varepsilon}-W_{y'})\right],
$$
and
$$
\Phi_{y,y'}(2,\varepsilon_1,\varepsilon_2) =E\left[f(W_y)f(W_{y'})(W_{y+\varepsilon_1}
-W_y)(W_{y'+\varepsilon_2}-W_{y'})\right].
$$
To end the proof, it is enough to assume that $f\in {\mathscr E}$ by denseness, and moreover, by approximating we can assume that $f$ is an infinitely differentiable function with compact support. It follows from~\eqref{sec7-eq7.9} that
\begin{align*}
\Phi_{y,y'}(1,\varepsilon)=
\sum_{j=1}^5\Psi_j(y,y',\varepsilon,\varepsilon),\quad
\Phi_{y,y'}(2,\varepsilon_1,\varepsilon_2)=
\sum_{j=1}^5\Psi_j(y,y',\varepsilon_1,\varepsilon_2),
\end{align*}
which give
\begin{align*}
\varepsilon_j&\Phi_{y,y'}(1,\varepsilon_i)-\varepsilon_i
\Phi_{y,y'}(2,\varepsilon_1,\varepsilon_2)\\
&=A_{y,y'}(1,\varepsilon_i,j)E\left[f''({W_y})f({W_{y'}})\right] +A_{y,y'}(2-1,\varepsilon_i,j)E\left[f'({W_y})f'({W_{y'}})\right]\\
&\quad+A_{y,y'}(3,\varepsilon_i,j)E\left[f({W_y})f''({W_{y'}}) \right]+A_{y,y'}(2-2,\varepsilon_i,j)E\left[f'({W_y})f'({W_{y'}}) \right]\\
&\quad+A_{y,y'}(4,\varepsilon,j)E\left[f({W_y})f({W_{y'}}\right] \end{align*}
with $i,j\in \{1,2\}$ and $i\neq j$, where
\begin{align*}
A_{y,y'}(1,\varepsilon,j):=\varepsilon_j &E\left[{W_y}(W_{y'+\varepsilon} -{W_{y'}})\right] E\left[{W_y}(W_{y+\varepsilon}-{W_y})\right]\\
&\qquad-\varepsilon E\left[{W_y}(W_{y'+\varepsilon_2} -{W_{y'}})\right]E\left[{W_y}(W_{y+\varepsilon_1}-{W_y})\right],\\
A_{y,y'}(2-1,\varepsilon,j):=
\varepsilon_j&E\left[{W_y}(W_{y'+\varepsilon}-{W_{y'}})\right] E\left[{W_{y'}}(W_{y+\varepsilon}-{W_y})\right]\\
&\qquad-\varepsilon E\left[{W_y}(W_{y'+\varepsilon_2}-{W_{y'}})\right] E\left[{W_{y'}}(W_{y+\varepsilon_1}-{W_y})\right],\\
A_{y,y'}(2-2,\varepsilon,j):=\varepsilon_j& E\left[{W_{y'}}(W_{y'+\varepsilon}-{W_{y'}})\right] E\left[{W_y}(W_{y+\varepsilon}-{W_y})\right]\\
&\qquad-\varepsilon E\left[{W_{y'}}(W_{y'+\varepsilon_2}-{W_{y'}})\right] E\left[{W_y}(W_{y+\varepsilon_1}-{W_y})\right],\\
A_{y,y'}(3,\varepsilon,j):=
\varepsilon_j&E\left[{W_{y'}}(W_{y'+\varepsilon}-{W_{y'}})\right] E\left[{W_{y'}}(W_{y+\varepsilon}-{W_y})\right]\\
&\qquad-\varepsilon E\left[{W_{y'}}(W_{y'+\varepsilon_2}-{W_{y'}})\right] E\left[{W_{y'}}(W_{y+\varepsilon_1}-{W_y})\right],\\
A_{y,y'}(4,\varepsilon,j):=\varepsilon_j& E\left[(W_{y+\varepsilon}-{W_y}) (W_{y'+\varepsilon}-{W_{y'}})\right]
-\varepsilon E\left[(W_{y+\varepsilon_1}-{W_y}) (W_{y'+\varepsilon_2}-{W_{y'}})\right]
\end{align*}
for all $\varepsilon_1,\varepsilon_2>0$ and $y,y'\in {\mathbb R}$. Now, we claim that the following convergences hold:
\begin{align}\label{sec3-eq3.12}
\frac1{\varepsilon_i^2\varepsilon_j}\int_0^t\int_0^t
\left\{\varepsilon_j\Phi_{y,y'}(1,\varepsilon_i)-\varepsilon_i
\Phi_{y,y'}(2,\varepsilon_1,\varepsilon_2)\right\}dsdr
\longrightarrow 0
\end{align}
with $i,j\in \{1,2\}$ and $i\neq j$, as $\varepsilon_1,\varepsilon_2\to 0$. We decompose that
\begin{align*}
\varepsilon_j&\Phi_{y,y'}(1,\varepsilon_i)-\varepsilon_i
\Phi_{y,y'}(2,\varepsilon_1,\varepsilon_2)\\
&=\left\{A_{y,y'}(1,\varepsilon_i,j) -A_{y,y'}(2-1,\varepsilon_i,j)\right\} E\left[f''({W_y})f({W_{y'}})\right]\\
&\qquad+A_{y,y'}(2-1,\varepsilon_i,j)
 \left\{E\left[f'({W_y})f'({W_{y'}})\right] +E\left[f''({W_y})f({W_{y'}})\right]\right\}\\
&\qquad+\left\{A_{y,y'}(3,\varepsilon_i,j) -A_{y,y'}(2-2,\varepsilon_i,j)\right\}E\left[f({W_y})f''({W_{y'}}) \right]\\
&\qquad+A_{y,y'}(2-2,\varepsilon_i,j)\left\{ E\left[f'({W_y})f'({W_{y'}})\right]+
E\left[f({W_y})f''({W_{y'}}) \right]\right\}\\
&\qquad+A_{y,y'}(4,\varepsilon,j)E\left[f({W_y})f({W_{y'}}\right] \end{align*}
with $i,j\in \{1,2\}$ and $i\neq j$. By symmetry we only need to introduce the convergence~\eqref{sec3-eq3.12} with $i=1$ and $j=2$.

{\bf Step I.} The following convergence holds:
\begin{align}\label{sec3-eq3.13}
\frac1{\varepsilon_1^2\varepsilon_2}\int_{I_x}\int_{I_x}
A_{y,y'}(4,\varepsilon_1,2)E\left[f({W_y})f(W_{y'})\right] dydy'
\longrightarrow 0,
\end{align}
as $\varepsilon_1,\varepsilon_2\to 0$. We have
\begin{align*}
A_{y,y'}(4,&\varepsilon_1,2)=\varepsilon_2 E\left[(W_{y+\varepsilon_1}-{W_y}) (W_{y'+\varepsilon_1}-{W_{y'}})\right]\\
&\qquad\qquad-\varepsilon_1E\left[(W_{y+\varepsilon_1}-{W_y}) (W_{y'+\varepsilon_2}-{W_{y'}})\right]\\
&=\varepsilon_2\left\{EW_{y+\varepsilon_1}W_{y'+\varepsilon_1}
-E{W_y}W_{y'+\varepsilon_1}-EW_{y+\varepsilon_1}{W_{y'}} +E{W_y}{W_{y'}}\right\}\\
&\qquad\qquad-\varepsilon_1\left\{EW_{y+\varepsilon_1}W_{y'+\varepsilon_2}
-E{W_y}W_{y'+\varepsilon_2}-EW_{y+\varepsilon_1}{W_{y'}} +E{W_y}{W_{y'}}\right\}\\
&=\frac1{2\sqrt{\pi}}\varepsilon_2\left(\int_0^t\frac2{\sqrt{r}}
e^{-\frac{(y-y')^2}{4r}}dr
-\int_0^t\frac1{\sqrt{r}}
e^{-\frac{(y-y'-\varepsilon_1)^2}{4r}}dr-\int_0^t\frac1{\sqrt{r}}
e^{-\frac{(y+\varepsilon_1-y')^2}{4r}}dr\right)\\
&\qquad\qquad-\frac1{2\sqrt{\pi}}\varepsilon_1\left(\int_0^t\frac1{\sqrt{r}}
e^{-\frac{(y-y'+\varepsilon_1-\varepsilon_2)^2}{4r}}dr
-\int_0^t\frac1{\sqrt{r}} e^{-\frac{(y-y'-\varepsilon_2)^2}{4r}}dr\right.\\
&\qquad\qquad\left.-\int_0^t\frac1{\sqrt{r}}
e^{-\frac{(y+\varepsilon_1-y')^2}{4r}}dr
+\int_0^t\frac1{\sqrt{r}}e^{-\frac{(y-y')^2}{4r}}dr\right).
\end{align*}
Consider the next function on ${\mathbb R}_{+}$ (see Section~\ref{sec2}):
$$
f(x)=x\int_{x}^\infty\frac1{s^2}e^{-\frac{s^2}2}ds =e^{-\frac{x^2}2}-x\int_{x}^\infty e^{-\frac{s^2}2}ds.
$$
Then we have
\begin{align*}
A_{y,y'}&(4,\varepsilon_1,2) =\frac{\sqrt{t}}{\sqrt{2\pi}}\varepsilon_2\left(
2f(\frac{y-y'}{\sqrt{2t}})-f(\frac{y-y'-\varepsilon_1}{\sqrt{2t}})
-f(\frac{y+\varepsilon_1-y'}{\sqrt{2t}})\right)\\
&-\frac{\sqrt{t}}{\sqrt{2\pi}} \varepsilon_1\left(f(\frac{y-y'+\varepsilon_1-\varepsilon_2}{ \sqrt{2t}})-f(\frac{y-y'-\varepsilon_2}{\sqrt{2t}})
-f(\frac{y+\varepsilon_1-y'}{\sqrt{2t}})+
f(\frac{y-y'}{\sqrt{2t}})\right).
\end{align*}
Notice that, by Taylor's expansion
$$
f(x)=1-\sqrt{\frac{\pi}{2}}x+\frac12x^2-\frac1{4!}x^4+o(x^4).
$$
One get
\begin{align*}
2f(\frac{y-y'}{\sqrt{2t}})&-f(\frac{y-y'-\varepsilon_1}{\sqrt{2t}})
-f(\frac{y-y'+\varepsilon_1}{\sqrt{2t}})\\
&=\frac1{4t}\left\{2(y-y')^2-(y-y'-\varepsilon_1)^2 -(y-y'+\varepsilon_1)^2\right\}\\
&\qquad-\frac1{4\times 4!t^2} \left\{2(y-y')^4-(y-y'-\varepsilon_1)^4 -(y-y'+\varepsilon_1)^4\right\}+\alpha_1\\
&=-\frac1{4t}\varepsilon_1^2+\frac1{4\times 4!t^2} \left(12(y-y')^2\varepsilon_1^2+2\varepsilon_1^4\right)+\alpha_1
\end{align*}
with $\alpha_1=\frac1{4t^2} o\left(12(y-y')^2\varepsilon_1^2+2\varepsilon_1^4\right)$ and
\begin{align*}
&f(\frac{y-y'+\varepsilon_1-\varepsilon_2}{ \sqrt{2t}})-f(\frac{y-y'-\varepsilon_2}{\sqrt{2t}})
-f(\frac{y-y'+\varepsilon_1}{\sqrt{2t}})+
f(\frac{y-y'}{\sqrt{2t}})\\
&=\frac1{4t}\left\{(y-y'+\varepsilon_1-\varepsilon_2)^2 -(y-y'-\varepsilon_2)^2 -(y-y'+\varepsilon_1)^2+(y-y')^2\right\}\\
&\qquad-\frac1{4\times 4!t^2} \left\{(y-y'+\varepsilon_1-\varepsilon_2)^4-(y-y'-\varepsilon_2)^4 -(y-y'+\varepsilon_1)^4+(y-y')^4\right\}+\alpha_2\\
&=-\frac1{4t}\varepsilon_1\varepsilon_2+\frac1{4\times 4!t^2} \left\{12(y-y')^2\varepsilon_1\varepsilon_2 +\varepsilon_1\varepsilon_2(\varepsilon_1-\varepsilon_2) \left(12(y-y')+
2\varepsilon_1-\varepsilon_2\right) \right\}+\alpha_2
\end{align*}
with $\alpha_2=\frac1{4t^2} o\left(12(y-y')^2\varepsilon_1\varepsilon_2 +\varepsilon_1\varepsilon_2(\varepsilon_1-\varepsilon_2) \left(12(y-y')+
2\varepsilon_1-\varepsilon_2\right)\right)$. It follows that
\begin{align*}
\frac1{\varepsilon_1^2\varepsilon_2}&|A_{y,y'}(4,\varepsilon_1,2)| \leq \frac{C}{t^{3/2}}\left(\varepsilon_1
+\frac{o(\varepsilon_1\varepsilon_2)}{\varepsilon_1\varepsilon_2}
+\frac{o(\varepsilon_1^2)}{\varepsilon_1^2}\right)(|x|^2+|x|+1).
\end{align*}
for all $0<\varepsilon_2<\varepsilon_1<1$ and $y,y'\in I_x$, which shows that the convergence~\eqref{sec3-eq3.12} holds since $f\in {\mathscr H}_t$.

{\bf Step II.} The following convergence holds:
\begin{equation}\label{sec3-eq3.14}
\begin{split}
\frac1{\varepsilon_1^2\varepsilon_2}\int_{I_x}\int_{I_x} &\left\{A_{y,y'}(1,\varepsilon_1,2) -A_{y,y'}(2-1,\varepsilon_1,2)\right\}\\
&\qquad\qquad\cdot E\left[f''({W_y})f({W_{y'}})\right] dydy'
\longrightarrow 0,
\end{split}
\end{equation}
as $\varepsilon_1,\varepsilon_2\to 0$. Keeping the notations in Step I, we have
\begin{align*}
f(\frac{y-y'-\varepsilon}{\sqrt{2t}})&-f(\frac{y-y'}{\sqrt{2t}}) =\frac1{\sqrt{\pi t}}\varepsilon-\frac{1}{4t}\left(2(y-y')-\varepsilon\right) \varepsilon\\
&-\frac1{4\times 4!t^2}\left(-4(y-y')^3+6\varepsilon(y-y')^2-4\varepsilon^3(y-y') +\varepsilon^3\right)\varepsilon
\end{align*}
for all $\varepsilon$ and
\begin{align*}
\Delta_1:&=\varepsilon_2\left(f(\frac{y-y'-\varepsilon_1}{\sqrt{2t}}) -f(\frac{y-y'}{\sqrt{2t}})\right)
-\varepsilon_1\left(
f(\frac{y-y'-\varepsilon_2}{\sqrt{2t}}) -f(\frac{y-y'}{\sqrt{2t}})\right)\\
&=\varepsilon_1\varepsilon_2(\varepsilon_1-\varepsilon_2)
\left\{\frac1{4t}-\frac{3(y-y')^2}{2\times 4!t^2}
+\frac{y-y'}{\times 4!t^2}(\varepsilon_1+\varepsilon_2)
-\frac1{4\times 4!t^2}(\varepsilon_1^2-\varepsilon_1\varepsilon_2+\varepsilon_2^2)
\right\}.
\end{align*}
It follows from~\eqref{sec2-eq2.7} that
\begin{align*}
|A_{y,y'}&(1,\varepsilon_1,2)-A_{y,y'}(2-1,\varepsilon_1,2)|\\
&=|\varepsilon_2E\left[{W_y}(W_{y'+\varepsilon_1} -{W_{y'}})\right] E\left[({W_y}-W_{y'})(W_{y+\varepsilon_1}-{W_y})\right]\\
&\qquad-\varepsilon_1 E\left[W_y(W_{y'+\varepsilon_2}-{W_{y'}})\right] E\left[(W_y-W_{y'})(W_{y+\varepsilon_1}-{W_y})\right]|\\
&=|E\left[(W_y-W_{y'})(W_{y+\varepsilon_1}-{W_y})\right]|\\
&\qquad\qquad\cdot\left|\varepsilon_2E\left[{W_y}(W_{y'+\varepsilon_1} -{W_{y'}})\right]-\varepsilon_1 E\left[W_y(W_{y'+\varepsilon_2}-{W_{y'}})\right]\right|\\
&=|E\left[(W_y-u_{y'})(W_{y+\varepsilon_1}-{W_y})\right]||\Delta_1|\\
&\leq C|y-y'|\varepsilon_1^2\varepsilon_2(\varepsilon_1-\varepsilon_2)
\left\{\frac1{t}+\frac1{t^2}({|x|^2}+|x|+1)\right\}
\end{align*}
for all $0<\varepsilon_2<\varepsilon_1<1$ and $y,y'\in I_x$, which implies that

\begin{align*}
\frac1{\varepsilon_1^2\varepsilon_2}&\int_{I_x}\int_{I_x} \left|A_{y,y'}(1,\varepsilon_1,2) -A_{y,y'}(2-1,\varepsilon_1,2)\right| |E\left[f''({W_y})f({W_{y'}})\right]|dydy'\\
&\leq C(\varepsilon_1-\varepsilon_2)
\left\{\frac1{t}+\frac1{t^2}({|x|^2}+|x|+1)\right\}
\int_{I_x}\int_{I_x}|y-y'||E\left[f''({W_y})f({W_{y'}})\right] |dydy'\\
&\leq C(\varepsilon_1-\varepsilon_2)
\left\{\frac1{t}+\frac1{t^2}({|x|^2}+|x|+1)\right\}
\|f\|^2_{{\mathscr H}_t}\longrightarrow 0,
\end{align*}
as $\varepsilon_1,\varepsilon_2\to 0$ since $f\in {\mathscr H}_t$. Similarly, we can show that the next convergence:
\begin{align}\label{sec3-eq3.15}
\frac1{\varepsilon_1^2\varepsilon_2}\int_{I_x}\int_{I_x} \left\{A_{y,y'}(3,\varepsilon_1,2) -A_{y,y'}(2-2,\varepsilon_1,2)\right\} E\left[f({W_y})f''({W_{y'}})\right]dydy'
\longrightarrow 0,
\end{align}
as $\varepsilon_1,\varepsilon_2\to 0$.

{\bf Step III.} The following convergence holds:
\begin{equation}\label{sec3-eq3.16}
\begin{split}
\frac1{\varepsilon_1^2\varepsilon_2}&\int_{I_x}\int_{I_x}
A_{y,y'}(2-1,\varepsilon_1,2)\\
&\qquad \cdot\left\{E\left[f'({W_y})f'({W_{y'}})\right] +E\left[f''({W_y})f({W_{y'}})\right]\right\}dydy'
\longrightarrow 0,
\end{split}
\end{equation}
as $\varepsilon_1,\varepsilon_2\to 0$. By Step II and Lemma~\ref{lem2.3}, we have
\begin{align*}
|A_{y,y'}(2-1,\varepsilon_1,2)|&= |E\left[{W_{y'}}(W_{y+\varepsilon_1}-{W_y})\right]|\\
&\qquad\cdot\left|
\varepsilon_2E\left[{W_y}(W_{y'+\varepsilon_1}-{W_{y'}})\right] -\varepsilon_1 E\left[{W_y}(W_{y'+\varepsilon_2}-{W_{y'}})\right] \right|\\
&=|E\left[{W_{y'}}(W_{y+\varepsilon_1}-{W_y})\right]||\Delta_1|\\
&\leq C\varepsilon_1^2\varepsilon_2(\varepsilon_1-\varepsilon_2)
\left\{\frac1{t}+\frac1{t^2}({|x|^2}+|x|+1)\right\}
\end{align*}
for all $0<\varepsilon_2<\varepsilon_1<1$ and $y,y'\in I_x$, which implies that
\begin{align*}
\frac1{\varepsilon_1^2\varepsilon_2}\int_{I_x}\int_{I_x}
&|A_{y,y'}(2-1,\varepsilon_1,2)| \left|E\left[f'({W_y})f'({W_{y'}})\right] +E\left[f''({W_y})f({W_{y'}})\right]\right|dydy'\\
&\leq
C(\varepsilon_1-\varepsilon_2)
\left\{\frac1{t}+\frac1{t^2}({|x|^2}+|x|+1)\right\}
\|f\|^2_{{\mathscr H}_t}\longrightarrow 0,
\end{align*}
as $\varepsilon_1,\varepsilon_2\to 0$ by~\eqref{sec3-eq3.999},  since $f\in {\mathscr H}_t$. Similarly, we can show that the next convergence:
\begin{equation}\label{sec3-eq3.17}
\begin{split}
\frac1{\varepsilon_1^2\varepsilon_2}&\int_{I_x}\int_{I_x}
A_{y,y'}(2-2,\varepsilon_1,2)\\
&\qquad\cdot\left\{E\left[f'({W_y})f'({W_{y'}})\right] +E\left[f({W_y})f''({W_{y'}})\right]\right\}dydy'
\longrightarrow 0,
\end{split}
\end{equation}
as $\varepsilon_1,\varepsilon_2\to 0$. Thus, we have proved the second statement.
\end{proof}

\begin{corollary}\label{cor4-1.1}
Let $f,f_1,f_2,\ldots \in {\mathscr H}_t$ such that $f_n\to f$ in ${\mathscr H}_t$. Then, the convergence
\begin{align}\label{sec7-eq7.6}
[f_n(W),W]^{(SQ)}_x\longrightarrow [f(W), W]^{(SQ)}_x
\end{align}
holds in $L^2(\Omega)$ for all $x\in {\mathbb R}$.
\end{corollary}

\section{The It\^o's formula of process $\{u(\cdot,x),x\in {\mathbb R}\}$}\label{sec4-1}
In this section, as a application of the previous section we discuss the It\^o calculus of the process $W=\{W_x=u(\cdot,x),x\in {\mathbb R}\}$ and fix a time parameter $t>0$. For a continuous processes $X$ admitting finite quadratic variation $[X,X]$, Russo and Vallois~\cite{Russo-Vallois2,Russo-Vallois3} have introduced the following It\^o formula:
$$
F(X_t)=F(X_0)+\int_0^tF'(X_s)d^{-}X_s+\frac1{2}\int_0^tF''(X_s) d[X,X]_s
$$
for all $F\in C^2({\mathbb R})$, where
$$
\int_0^tF'(X_s)d^{-}X_s:={\rm ucp}\lim_{\varepsilon\downarrow 0}\frac1{\varepsilon}\int_0^tF'(X_s) \left(X_{s+\varepsilon}-X_s\right)ds
$$
is called the forward integral, where the notation ${\rm ucp}\lim$ denotes the uniform convergence in probability on each compact interval. We refer to Russo and Vallois~\cite{Russo-Vallois2,Russo-Vallois3} and the references therein for more details of stochastic calculus of continuous processes with finite quadratic variations. It follows from the previous section (the quadratic variation of $\{W_x,x\in {\mathbb R}\}$ is equal to $|x|$ for all $x\in {\mathbb R}$) that
\begin{equation}\label{sec4-1-eq4.1}
F(W_x)=F(W_0)+\int_{I_x}F'(W_y)d^{-}W_y+\frac1{2}\int_{I_x}F''(W_y)dy
\end{equation}
for all $F\in C^2({\mathbb R})$. Thys, by smooth approximating, we have that the next It\^o type formula.
\begin{theorem}\label{th4-1.1}
Let $f\in {\mathscr H}_t$ be left continuous. If $F$ is an absolutely continuous function with the derivative $F'=f$,
then the following It\^o type formula holds:
\begin{equation}\label{sec4-1-eq4.1-1}
F(W_x)=F(W_0)+\int_{I_x}f(W_y)d^{-}W_y +\frac1{2}[f(W),W]^{(SQ)}_x.
\end{equation}
\end{theorem}
Clearly, this is an analogue of F\"ollmer-Protter-Shiryayev's formula. It is an improvement in terms of the hypothesis on $f$ and it is also quite interesting itself. Some details and more works could be found in Eisenbaum~\cite{Eisen1,Eisen2}, Feng--Zhao~\cite{Feng,Feng3}, F\"ollmer {\it et al}~\cite{Follmer}, Moret--Nualart~\cite{Moret}, Peskir~\cite{Peskir1}, Rogers--Walsh~\cite{Rogers2},
Russo--Vallois~\cite{Russo2,Russo-Vallois2,Russo-Vallois3},
Yan {\it et al}~\cite{Yan7,Yan2}, and the references therein. It is well-known that when $W$ is a semimartingale, the forward integral coincides with the It\^o integral. However, the following theorem points out that the two integrals are coincident for the process $W=\{W_x,x\in {\mathbb R}\}$. But, $W=\{W_x,x\in {\mathbb R}\}$ is not a semimartingale.
\begin{theorem}\label{th4-1.2}
Let $f$ be left continuous. If $F$ is an absolutely continuous function with the derivative $F'=f$ satisfying the condition
\begin{equation}\label{sec4-1-eq4.8011}
|F(y)|,|f(y)|\leq Ce^{\beta {y^2}},\quad y\in {\mathbb R}
\end{equation}
with $0\leq \beta<\frac{\sqrt{\pi}}{4\sqrt{t}}$, then the following It\^o type formula holds:
\begin{equation}\label{sec4-1-eq4.2}
F(W_x)=F(W_0)+\int_{I_x}f(W_y)\delta W_y +\frac1{2}[f(W_{\cdot}),W_{\cdot}]^{(SQ)}_x.
\end{equation}
\end{theorem}
According to the two theorems above we get the next relationship:
\begin{equation}\label{sec4-1-eq4.3}
\int_{I_x}f(W_y)\delta W_y=\int_{I_x}f(W_y)d^{-}W_y,
\end{equation}
if $f$ satisfies the growth condition~\eqref{sec4-1-eq4.8011}.

\begin{proof}[Proof of Theorem~\ref{th4-1.1}]
If $f\in C^1({\mathbb R})$, then this is It\^o's formula since
$$
\left[f(W),W\right]^{(SQ)}_x=\int_{I_x}f'(W_y)dy.
$$
For $f\not\in C^1({\mathbb R})$, by a localization argument we may
assume that the function $f$ is uniformly bounded. In fact, for any
$k\geq 0$ we may consider the set
$$
\Omega_k=\left\{\sup_{x\in {\mathbb R}}|W_x|<k\right\}
$$
and let $f^{[k]}$ be a measurable function such that $f^{[k]}=f$ on
$[-k,k]$ and such that $f^{[k]}$ vanishes outside. Then $f^{[k]}$ is uniformly bounded and $f^{[k]}\in {\mathscr H}_t$ for every $k\geq 0$. Set $\frac{d}{dx}F^{[k]}=f^{[k]}$ and $F^{[k]}=F$ on $[-k,k]$. If the theorem is true for all uniformly bounded functions on ${\mathscr H}_t$, then we get the desired formula
$$
F^{[k]}(W_x)=F^{[k]}(W_0)+\int_{I_x}
f^{[k]}(W_y)d^{-}u_y+\frac12[f^{[k]}(W),W]^{(SQ)}_x
$$
on the set $\Omega_k$. Letting $k$ tend to infinity we deduce the It\^o formula~\eqref{sec4-1-eq4.1}.

Let now $F'=f\in {\mathscr H}_t$ be uniformly bounded and left
continuous. Consider the function $\zeta$ on ${\mathbb R}$ by
\begin{equation}
\zeta(x):=
\begin{cases}
ce^{\frac1{(x-1)^2-1}}, &{\text { $x\in (0,2)$}},\\
0, &{\text { otherwise}},
\end{cases}
\end{equation}
where $c$ is a normalizing constant such that $\int_{\mathbb
R}\zeta(x)dx=1$. Define the mollifiers
\begin{equation}\label{sec4-eq00-4}
\zeta_n(x):=n\zeta(nx),\qquad n=1,2,\ldots
\end{equation}
and the sequence of smooth functions
$$
F_n(x):=\int_{\mathbb R}F(x-{y})\zeta_n(y)dy,\quad x\in {\mathbb R}.
$$
Then $F_n\in C^\infty({\mathbb R})$ for all
$n\geq 1$ and the It\^{o} formula
\begin{equation}\label{sec3-eq3-Ito-1}
F_n(W_x)=F_n(W_0)+\int_{I_x}f_n(W_y)d^{-}W_y+
\frac12\int_0^tf'_n(W_y)dy
\end{equation}
holds for all $n\geq 1$, where $f_n=F_n'$. Moreover, by using Lebesgue's dominated convergence theorem, one can prove that as $n\to \infty$, for each $x$,
$$
F_n(x)\longrightarrow F(x),\quad f_n(x)\longrightarrow f(x),
$$
and $\{f_n\}\subset {\mathscr H}_t$, $f_n\to f$ in ${\mathscr H}_t$. It follows that
\begin{align*}
\frac12\int_0^tf'_n(W_y)dy=\left[f_n(W),W\right]^{(SQ)}_x
\longrightarrow \left[f(W),W\right]^{(SQ)}_x
\end{align*}
and
$$
f_n(W_x)\longrightarrow f(W_x)
$$
in $L^2(\Omega)$ by Corollary~\ref{cor4-1.1}, as $n$ tends to infinity. It
follows that
\begin{align*}
\int_{I_x}f_n(W_y)d^{-}W_y&=F_n(W_y)-F_n(W_0)-
\frac12[f_n(W),W]^{(SQ)}_x\\
&\longrightarrow F(W_y)-F_n(W_0)-
\frac12[f(W),W]^{(SQ)}_x
\end{align*}
in $L^2(\Omega)$, as $n$ tends to infinity. This completes the proof since the integral is closed in $L^2(\Omega)$.
\end{proof}

Now, similar to proof of Theorem~\ref{th4-1.1} one can introduce Theorem~\ref{th4-1.2}. But, we need to give the following standard It\^o type formula:
\begin{equation}\label{sec4-1-eq4.8}
F(W_x)=F(W_0)+\int_{I_x}F'(W_y)\delta W_y +\frac1{2}\int_{I_x}F''(W_y)dy
\end{equation}
for all $F\in C^2({\mathbb R})$ satisfying the condition
\begin{equation}
|F(y)|,|F'(y)|,|F''(y)|\leq Ce^{\beta {y^2}},\quad y\in {\mathbb R}
\end{equation}
with $0\leq \beta<\frac{\sqrt{\pi}}{4\sqrt{t}}$. It is important to note that one have given a standard It\^o formula for a large class of Gaussian processes in Al\'os {\em et al}~\cite{Nua1}. However, the process $x\mapsto u(\cdot,x)$ does not satisfy the condition in Al\'os {\em et al}~\cite{Nua1} since
\begin{align*}
E\left[u(t,x)^2\right]=\sqrt{\frac{t}{\pi}},\quad \frac{d}{dx}E\left[u(t,x)^2\right]=0
\end{align*}
for all $t\geq 0$ and $x\in {\mathbb R}$. So, we need to give the proof of the formula~\eqref{sec4-1-eq4.8} in order to prove Theorem~\ref{th4-1.2}.
\begin{lemma}\label{lem4-1.1}
Let $x\in {\mathbb R}$ and let $x^n_j=\frac{jx}{n}; j=0,1,\ldots,n$. Then we have
\begin{equation}
\sum_{j=1}^n\left(W_{x^n_j}-W_{x^n_{j-1}}\right)^2\longrightarrow |x|,
\end{equation}
in $L^2$, as $n$ tends to infinity.
\end{lemma}
\begin{proof}
Similar to proof of Proposition~\ref{prop4.1} one can introduce the lemma, so we omit it.
\end{proof}
\begin{proof}[Proof of~\eqref{sec4-1-eq4.8}]
Let us fix $x\in {\mathbb R}$ and let $\pi\equiv \{x^n_j=\frac{jx}{n}; j=0,1,\ldots,n\}$ be a partition of $I_x$. Clearly, the growth condition~\eqref{sec4-1-eq4.2} implies that
\begin{equation}\label{sec4-1-eq4.8100}
E\left[\sup_{x\in {\mathbb R}}|G(W_x)|^p\right]\leq
c^pE\left[e^{p\beta\sup_{x\in {\mathbb R}}|W_x|}\right]<\infty
\end{equation}
for some constant $c>0$ and all $p<\frac{\sqrt{\pi}}{2\beta\sqrt{t}}$, where $G\in \{F,F',F''\}$. In particular, the estimate~\eqref{sec4-1-eq4.8100} holds for $p=2$. Using Taylor expansion, we have

\begin{equation}\label{sec3-eq3.3}
\begin{split}
F(W_x)&=F(W_0)+\sum^{n}_{j=1}
F^{'}(W_{x^n_{j-1}})(W_{x^n_{j}}-W_{x^n_{j-1}})\\
&\qquad+\frac{1}{2}\sum^{n}_{j=1}F^{''}
(W_{j}(\theta_j))(W_{x^n_{j}}-W_{x^n_{j-1}})^{2}\\
&\equiv F(W_0)+I^n +J^n
\end{split}
\end{equation}
where
$W_{j}(\theta_j)=W_{x^n_{j-1}}+\theta_j(W_{x^n_{j}}-W_{x^n_{j-1}})$
with $\theta_j\in (0,1)$ being a random variable. By~\eqref{sec2-eq2.1} we have
\begin{align*}
I^n&=\sum^n_{j=1}F^{'}(W_{x^n_{j-1}})(\delta^{t}(1_{(x^n_{j-1},
x^n_{j}]}))\\
&=\delta^{t}\left(\sum^n_{j=1}f^{'}(W_{x^n_{j-1}})1_{(x^n_{j-1},
x_{j}]}(\cdot)\right)
+\sum^n_{j=1}F^{''}(W_{x^n_{j-1}})\langle1_{(0,x^n_{j-1}]},
1_{(x^n_{j-1}, x^n_{j}]}\rangle_{\mathcal{H}_t}\\
& \equiv I^{n}_1+I^{n}_2.
\end{align*}
Now, in order to end the proof we claim that the following convergences in $L^2$ hold:
\begin{align}\label{sec4-1-eq4.11}
I^n_2\longrightarrow-\frac12\int^t_0F^{''}(W_y)dy,\\
I^{n}_1\longrightarrow\int_{I_x}F'(W_y)\delta W_y,\\
J^n\longrightarrow\int^t_0F^{''}(W_y)dy,
\end{align}
as $n$ tends to infinity.

To prove the first convergence, it is enough to establish that
\begin{align*}
\Lambda_n:=E\left|I^n_2+\frac12\sum_{j=1}^n F^{''}(W_{x^n_j})(x^n_j-x^n_{j-1})\right|^2\longrightarrow 0,
\end{align*}
as $n$ tends to infinity. By Minkowski inequality we have
\begin{align*}
\sqrt{\Lambda_n}&=\left(E\left|\sum^n_{j=1}F^{''}(W_{x^n_{j-1}}) \left\{\langle1_{(0,x^n_{j-1}]},
1_{(x^n_{j-1}, x^n_{j}]}\rangle_{\mathcal{H}_t}+\frac12 (x^n_j-x^n_{j-1})\right\}\right|^2\right)^{1/2}\\
&\leq C\sum^n_{j=1}\left|\langle1_{(0,x^n_{j-1}]},
1_{(x^n_{j-1}, x^n_{j}]}\rangle_{\mathcal{H}_t}+\frac12 (x^n_j-x^n_{j-1})\right|\\
&=C\sum^n_{j=1}\left|\frac1{2\sqrt{\pi}}\int_0^t \frac1{\sqrt{r}}\left(1-e^{-\frac1{4r}(x^n_j-x^n_{j-1})^2}\right)dr -\frac12 (x^n_j-x^n_{j-1})\right|\\
&=C\sum_{j=1}^n|x^n_j-x^n_{j-1}|
\left|\frac1{\sqrt{2\pi}} \int_{\frac{x^n_j-x^n_{j-1}}{\sqrt{2t}}}^{\infty} \frac1{s^2}\left(1-e^{-\frac{s^2}{2}}\right)ds-\frac12\right|\\
&=C\sum_{j=1}^n|x^n_j-x^n_{j-1}|
\left|\frac1{\sqrt{2\pi}} \int_0^{\frac{x^n_j-x^n_{j-1}}{\sqrt{2t}}} \frac1{s^2}\left(1-e^{-\frac{s^2}{2}}\right)ds\right|\\
&=C|x|\frac1{\sqrt{2\pi}} \int_0^{\frac{|x|}{n\sqrt{2t}}} \frac1{s^2}\left(1-e^{-\frac{s^2}{2}}\right)ds\longrightarrow 0,
\end{align*}
as $n$ tends to infinity by the fact
$$
\frac1{\sqrt{2\pi}} \int_0^\infty \frac1{s^2}\left(1-e^{-\frac{s^2}{2}}\right)ds=\frac1{\sqrt{2\pi}} \int_0^\infty e^{-\frac{s^2}{2}}ds=\frac12.
$$

Now, we prove the third convergence. We have
\begin{align*}
\Lambda_n(2):&=E\left|J^n-\int_{I_x} F^{''}(W_{y})dy\right|\\
&=E\left|\frac{1}{2}\sum^{n}_{j=1}F^{''}
(W_{j}(\theta_j))(W_{x^n_{j}}-W_{x^n_{j-1}})^{2}-\int_{I_x} F^{''}(W_{y})dy\right|.
\end{align*}
Suppose that $n\geq m$, and for any $j=1,\ldots,n$ let us denote by $x^{m(n)}_j$ the point of the $m$th partition that is closer to $x^n_j$ from the left. Then we obtain
\begin{align*}
\Lambda_n(2)&\leq \frac{1}{2}E\left|\sum^{n}_{j=1}\left(F^{''}
(W_{j}(\theta_j))-F^{''}(W_{x^{m(n)}_j})\right) (W_{x^n_{j}}-W_{x^n_{j-1}})^{2}\right|\\
&\qquad+\frac12E\left|\sum_{k=1}^mF^{''}(W_{x^{m(n)}_j})\sum_{
\{j:x^{m(n)}_{k-1}\leq x^{m(n)}_{j-1}<x^{m(n)}_{k}\}}
\left((W_{x_j^n}-W_{x^n_{j-1}})^2-(x_j^n-x^n_{j-1})\right)\right|\\
&\qquad+E\left|\sum_{k=1}^m\int_{x^n_{k-1}}^{x^n_k} \left(F^{''}(W_{x^{m(n)}_j})-F^{''}(W_y)\right)dy\right|\\
&\equiv \frac{1}{2}\Lambda_n(2,1)+\frac{1}{2}\Lambda_n(2,2)+\Lambda_n(2,3).
\end{align*}
Clearly, we have that $\Lambda_n(2,2)\to 0$ ($n,m\to \infty$) by Lemma~\ref{lem4-1.1} and the estimate~\eqref{sec4-1-eq4.8100},
\begin{align*}
\Lambda_n(2,3)&\leq |x|E\sup_{|z-y|\leq \frac{|x|}{m}}|F^{''}
(W_{z})-F^{''}(W_{y})|\longrightarrow 0\quad(m\to \infty)
\end{align*}
and

\begin{align*}
\Lambda_n(2,1)&=E\left|\sum^{n}_{j=1}\left(F^{''}
(W_{j}(\theta_j))-F^{''}(W_{x^{m(n)}_j})\right) (W_{x_{j}}-W_{x_{j-1}})^{2}\right|\\
&\leq CE\left\{
\sup_{|z-y|\leq \frac{|x|}{n}}|F^{''}
(W_{z})-F^{''}(W_{y})|
\sum^{n}_{j=1}(W_{x^n_{j}}-W_{x^n_{j-1}})^{2}\right\}\\
&\leq C\left\{E
\sup_{|z-y|\leq \frac{|x|}{n}}|F^{''}
(W_{z})-F^{''}(W_{y})|^2
E\left(\sum^{n}_{j=1}(W_{x^n_{j}}-W_{x^n_{j-1}})^{2}\right)^2 \right\}^{1/2}\\
&\leq C|x|\left\{E\sup_{|z-y|\leq \frac{|x|}{n}}|F^{''}
(W_{z})-F^{''}(W_{y})|^2\right\}^{1/2}\longrightarrow 0 \quad(n\to \infty)
\end{align*}
by~\eqref{sec2-eq2.2} and the estimate~\eqref{sec4-1-eq4.8100}. Thus, we obtain the third convergence, i.e., $J^n\to \int^t_0F^{''}(W_y)dy$ in $L^1$.

Finally, to end the proof we show that the second convergence:
$$
I^{n}_1=\delta^{t}\left(\sum^n_{j=1}F^{'}(W_{x^n_{j-1}}) 1_{(x^n_{j-1},
x^n_{j}]}(\cdot)\right)\longrightarrow\int_{I_x}F'(W_y)\delta W_y \quad (n\to \infty).
$$
We need to show that
$$
A_n:=\sum^n_{j=1}F^{'}(W_{x^n_{j-1}})1_{(x^n_{j-1},x^n_{j}]}(\cdot)
\longrightarrow F^{'}(W_{\cdot})1_{I_x}(\cdot)
$$
in $L^2(\Omega;{\mathcal H}_t)$, as $n$ tends to infinity. We have
\begin{align*}
E\|A_n&-F^{'}(W_{\cdot})1_{I_x}(\cdot)\|^2_{{\mathcal H}_t}\\
&=E\Bigl\|\sum^n_{j=1}\left(F^{'}(W_{x_{j-1}})-F^{'}(W_{\cdot}) \right) 1_{(x^n_{j-1},x^n_{j}]}(\cdot)
\Bigr\|^2_{|\mathcal{H}_t|}\\
&=E\sum^{n}_{j,l=1}\int^{x^n_{j}}_{x^n_{j-1}}
\int_{x^n_{l-1}}^{x^n_{l}}\left|F^{'}(W_{x^n_{j-1}})-F^{'}(W_y) \right|\left|F^{'}(W_{x^n_{j-1}})-F^{'}(W_z)\right|\phi(y,z)dydz\\
&\leq E\sup_{|y-z|\leq \frac{|x|}n}
\left|F^{'}(W_y)-F^{'}(W_z)\right|^2
\int_{I_x}\int_{I_x}\phi(y,z)dydz\\
&=E\sup_{|y-z|\leq \frac{|x|}n}
\left|F^{'}(W_y)-F^{'}(W_z)\right|^2
\int_{I_x}\int_{I_x}\frac1{2\sqrt{\pi t}}e^{-\frac{(y-z)^2}{4t}} dydz\\
&\leq E\sup_{|y-z|\leq \frac{|x|}n}
\left|F^{'}(W_y)-F^{'}(W_z)\right|^2
\int_{I_x}dz\int_{\mathbb R}\frac1{2\sqrt{\pi t}}e^{-\frac{(y-z)^2}{4t}}dy\\
&\leq  |x|E\sup_{|y-z|\leq \frac{|x|}n}
\left|F^{'}(W_y)-F^{'}(W_z)\right|^2\longrightarrow 0,
\end{align*}
as $n$ tends to infinity by the estimate~\eqref{sec4-1-eq4.8100}, which implies that
$$
A_n:=\sum^n_{j=1}F^{'}(W_{x^n_{j-1}})1_{(x^n_{j-1},x^n_{j}]}(\cdot)
\longrightarrow F^{'}(W_{\cdot})1_{I_x}(\cdot)
$$
in $L^2(\Omega;{\mathcal H}_t)$, as $n$ tends to infinity. The above steps prove that
\begin{align*}
I^n_1=F(W_x)-F(W_0)-I^n_2-J^n\longrightarrow F(W_x)-F(W_0)-\frac12\int_{I_x}F''(W_y)dy
\end{align*}
in $L^2(\Omega)$, as $n$ tends to infinity. This completes the proof since the integral $\int_0^\cdot u_s\delta W_s$ is closed in $L^2(\Omega)$.
\end{proof}

\section{The Bouleau-Yor identity of $\{u(\cdot,x),x\geq 0\}$}\label{sec4-2}
In this section, we consider the local time of the process $\{W_x=u(\cdot,x),x\geq 0\}$. Our main object is to prove that the integral
$$
\int_{\mathbb R}g(a){\mathscr L}^t(x,da)
$$
is well-defined and that the identity
\begin{equation}\label{sec4-2-eq0}
\int_{\mathbb R}g(a){\mathscr L}^t(x,da)=-[g(u(t,\cdot)),u(t,\cdot)]_x^{(SQ)}
\end{equation}
holds for all $g\in {\mathscr H}_t$ and $t>0$, where
$$
{\mathscr L}^t(x,a)=\int_{I_x}\delta(u(t,y)-a)dy.
$$
is the local time of $\{W_x=u(\cdot,x),x\geq 0\}$. The identity~\eqref{sec4-2-eq0} is called the Bouleau-Yor identity. More works for this can be found in Bouleau-Yor~\cite{Bouleau}, Eisenbaum~\cite{Eisen1,Eisen2}, F\"ollmer {\it
et al}~\cite{Follmer}, Feng--Zhao~\cite{Feng,Feng3},
Peskir~\cite{Peskir1}, Rogers--Walsh~\cite{Rogers2},
Yan {\it et al}~\cite{Yan7,Yan2}, and the references therein.

Recall that for any closed interval $I\subset {\mathbb R}_{+}$ and for any $a\in {\mathbb R}$, the local time $L(a,I)$ of $u$ is defined as the density of the occupation measure $\mu_I$ defined by
$$
\mu_I(A)=\int_I1_A(W_x)dx
$$
It can be shown (see Geman and Horowitz~\cite{Geman}, Theorem 6.4) that the following occupation density formula holds:
$$
\int_Ig(W_x,x)dx=\int_{\mathbb R}da\int_Ig(a,x)L(a,dx)
$$
for every Borel function $g(a,x)\geq 0$ on $I\times {\mathbb R}$. Thus, some estimates in Section~\ref{sec2} and Theorem 21.9 in Geman-Horowitz~\cite{Geman} together imply that the following result holds.

\begin{corollary}
The local time ${\mathscr L}^t(a,x):=L(a, [0,x])$ of $W=\{W_x=u(\cdot,x),x\geq 0\}$ exists and ${\mathscr L}^t\in L^2(\lambda\times P)$ for all $x\geq 0$ and $(a,x)\mapsto {\mathscr L}^t(a,x)$ is jointly continuous, where $\lambda$ denotes Lebesgue measure. Moreover, the occupation formula
\begin{equation}\label{sec4-2-eq1}
\int_0^t\psi(W_x,x)dx=\int_{\mathbb R}da\int_0^t\psi(a,x) {\mathscr L}^t(a,dx)
\end{equation}
holds for every continuous and bounded function $\psi(a,x):{\mathbb
R}\times {\mathbb R}_{+}\rightarrow {\mathbb R}$ and any $x\geq 0$.
\end{corollary}

\begin{lemma}
For any $f_\triangle=\sum_jf_j1_{(a_{j-1},a_j]}\in {\mathscr E}$, we define
$$
\int_{\mathbb R}f_\triangle(y){\mathscr
L}^t(dy,x):=\sum_jf_j\left[{\mathscr L}^t(a_j,x)-{\mathscr
L}^{t}(a_{j-1},x)\right].
$$
Then the integral is well-defined and
\begin{equation}\label{sec4-2-eq2}
\int_{\mathbb R}f_{\Delta}(y)\mathscr{L}^t(dy,x)=
-\bigl[f_\triangle(W),W\bigr]^{(SQ)}_x
\end{equation}
almost surely, for all $x\geq 0$.
\end{lemma}
\begin{proof}
For the function $f_\triangle(y)=1_{(a,b]}(y)$ we define the sequence of smooth functions $f_n,\;n=1,2,\ldots$ by
\begin{align}
f_n(y)&=\int_{\mathbb
R}f_\triangle(y-z)\zeta_n(z)dz=\int_a^b\zeta_n(y-z)dz
\end{align}
for all $y\in \mathbb R$, where $\zeta_n,n\geq 1$ are the so-called mollifiers given in~\eqref{sec4-eq00-4}. Then $\{f_n\}\subset
C^{\infty}({\mathbb R})\cap {\mathscr H}_t$ and $f_n$ converges to $f_\triangle$ in ${\mathscr H}_t$, as $n$ tends to infinity. It follows from the occupation formula that
\begin{align*}
[f_n(W),W]^{(SQ)}_x& =\int_0^xf'_n(W_y)dy\\
&=\int_{\mathbb R}f_n'(y){\mathscr L}^t(y,x)dy=\int_{\mathbb
R}\left(\int_a^b\zeta_n'(y-z)dz\right){\mathscr L}^t(y,x)dy\\
&=-\int_{\mathbb R}{\mathscr
L}^t(y,x)\left(\zeta_n(y-b)-\zeta_n(y-a)\right)dy\\
&=\int_{\mathbb R}{\mathscr L}^{t}(y,x)\zeta_n(y-a)dy
-\int_{\mathbb R}{\mathscr L}^{t}(y,x)\zeta_n(y-b)dy\\
&\longrightarrow {\mathscr L}^{t}(a,x)-{\mathscr L}^{t}(b,x)
\end{align*}
almost surely, as $n\to \infty$, by the continuity of $y\mapsto
{\mathscr L}^{t}(y,x)$. On the other hand, we see also that there exists a subsequence $\{f_{n_k}\}$ such that
$$
[f_{n_k}(W),W]^{(SQ)}_x\longrightarrow [1_{(a,b]}(W),W]^{(SQ)}_x
$$
for all $x\geq 0$, almost surely, as $k\to \infty$ since $f_n$ converges to $f_\triangle$ in ${\mathscr H}_t$. It follows that
$$
[1_{(a,b]}(W),W]^{(SQ)}_x=\left({\mathscr L}^{t}(a,x)-{\mathscr L}^{t}(b,x)
\right)
$$
for all $x\geq 0$, almost surely. Thus, the identity
$$
\sum_jf_j[{\mathscr L}^{t}(a_j,x)-{\mathscr L}^{t}(a_{j-1},x)]=
-[f_\triangle(W),W]^{(SQ)}_x
$$
follows from the linearity property, and the lemma follows.
\end{proof}

As a direct consequence of the above lemma, for every $f\in {\mathscr H}_t$ if
$$
\lim_{n\to \infty}f_{\triangle,n}(y)=\lim_{n\to
\infty}g_{\triangle,n}(x)=f(y)
$$
in ${\mathscr H}_t$, where $\{f_{\triangle,n}\},\{g_{\triangle,n}\}\subset {\mathscr E}$, we then have that
\begin{align*}
\lim_{n\to \infty}\int_{\mathbb
R}&f_{\triangle,n}(y){\mathscr L}^{t}(y,x)dy
=-\lim_{n\to \infty}[f_{\triangle,n}(W),W]^{(SQ)}_x= -[f(W),W]^{(SQ)}_x\\
&=-\lim_{n\to \infty}[g_{\triangle,n}(W),W]^{(SQ)}_x=\lim_{n\to
\infty}\int_{\mathbb R}g_{\triangle,n}(y){\mathscr L}^{t}(y,x)dy
\end{align*}
in $L^2(\Omega)$. Thus, by the denseness of ${\mathscr
E}$ in ${\mathscr H}_t$ we can define
$$
\int_{\mathbb R}f(y){\mathscr L}^{t}(dy,x):=\lim_{n\to
\infty}\int_{\mathbb R}f_{\triangle,n}(y){\mathscr L}^{t}(dy,x)
$$
for any $f\in {\mathscr H}_t$, where $\{f_{\triangle,n}\}\subset
{\mathscr E}$ and
$$
\lim_{n\to \infty}f_{\triangle,n}=f
$$
in ${\mathscr H}_t$. These considerations are enough to prove the following theorem.
\begin{theorem}\label{th4-2.1}
For any $f\in {\mathscr H}_t$, the integral
$$
\int_{\mathbb R}f(y){\mathscr L}^{t}(dy,x)
$$
is well-defined in $L^2(\Omega)$ and the Bouleau-Yor type identity
\begin{equation}\label{sec4-2-eq4}
[f(W),W]^{(SQ)}_x=-\int_{\mathbb R}f(y)\mathscr{L}^{t}(dy,x)
\end{equation}
holds, almost surely, for all $x\geq 0$.
\end{theorem}

\begin{corollary}[Tanaka formula]
For any $a\in {\mathbb R}$ we have
\begin{align*}
(W_x-a)^{+}=(W_0-a)^{+}+\int_0^x{1}_{\{W_y>a\}} \delta W_y+\frac12{\mathscr L}^{t}(a,x),\\
(W_x-a)^{-}=(W_0-a)^{-}-\int_0^x{1}_{\{W_y<a\}}
\delta W_y+\frac12{\mathscr L}^{t}(a,x),\\
|W_x-a|=|W_0-a|+\int_0^x{\rm sign}(W_x-a)\delta W_y+{\mathscr L}^{t}(a,x).
\end{align*}
\end{corollary}
\begin{proof}
Take $F(y)=(y-x)^{+}$. Then $F$ is absolutely continuous and
$$
F(x)=\int_{-\infty}^y1_{(x,\infty)}(y)dy.
$$
It follows from the identity~\eqref{sec4-2-eq2} and It\^o's formula~\eqref{th4-1.2} that
\begin{align*}
{\mathscr
L}^{t}(a,x)&=[1_{(a,+\infty)}(W),W]^{(SQ)}_x\\
&=2(W_x-a)^{+}-2(-a)^{+}-2 \int_0^x{1}_{\{W_y>a\}}\delta W_y
\end{align*}
for all $x\geq 0$, which gives the first identity. In the same
way one can obtain the second identity, and by subtracting the last identity from the previous one, we get the third identity.
\end{proof}

According to Theorem~\ref{th4-2.1}, we get an analogue of the It\^o formula (Bouleau-Yor type formula).
\begin{corollary}\label{cor4-2.3}
Let $f\in {\mathscr H}_t$ be a left continuous function with right limits. If $F$ is an absolutely continuous function with $F'=f$, then the following It\^o type formula holds:
\begin{equation}\label{sec4-2-eq5}
F(W_x)=F(W_0)+\int_0^xf(W_y)\delta W_y-\frac12\int_{\mathbb R}f(y){\mathscr L}^{t}(dy,x).
\end{equation}
\end{corollary}
Recall that if $F$ is the difference of two convex functions, then
$F$ is an absolutely continuous function with derivative of bounded
variation. Thus, the It\^o-Tanaka formula
\begin{align*}
F(W_x)&=F(0)+\int_0^xF^{'}(W_y)\delta W_y+\frac12\int_{\mathbb R}{\mathscr L}^{t}(y,x)F''(dy)\\
&\equiv F(0)+\int_0^xF^{'}(W_y)\delta W_y-\frac12\int_{\mathbb R}F'(y){\mathscr L}^{t}(dy,x)
\end{align*}
holds.

\section{The quadratic covariation of process $\{u(t,\cdot),t\geq 0\}$}\label{sec6}
In this section, we study the existence of the
PQC $[f(u(\cdot,x)),u(\cdot,x)]^{(TQ)}$. Recall that
$$
I_\varepsilon^2(f,x,t)=\frac1{\sqrt{\varepsilon}}\int_0^t \left\{f(u(s+\varepsilon,x))-f(u(s,x))\right\}
(u(s+\varepsilon,x)-u(s,x))\frac{ds}{2\sqrt{s}}
$$
for $\varepsilon>0,t\geq 0$ and $x\in {\mathbb R}$, and
\begin{equation}\label{sec6-eq6.1}
[f(u(\cdot,x)),u(\cdot,x)]^{(TQ)}_t=\lim_{\varepsilon\downarrow
0}I_\varepsilon^2(f,x,t),
\end{equation}
provided the limit exists in probability. In this section, we study some analysis questions of the process $\{u(t,\cdot),t\geq 0\}$ associated with the quadratic covariation $[f(u(\cdot,x)),u(\cdot,x)]^{(TQ)}$, and the researches include the existence of the
PQC $[f(u(\cdot,x)),u(\cdot,x)]^{(TQ)}$, the It\^o and Tanaka formulas. Recall that
\begin{align*}
E\left[u(t,x)^2\right]=\sqrt{\frac{t}{\pi}}
\end{align*}
for all $t\geq 0$ and $x\in {\mathbb R}$. Denote
$$
B_t=u(t,\cdot),\quad t\in [0,T].
$$
It follows from Al\'os {\em et al}~\cite{Nua1} that the It\^o formula
\begin{equation}\label{sec6-eq6.2}
f(B_t)=f(0)+\int_0^tf'(B_s)\delta B_s+\frac1{2\sqrt{2}}\int_0^tf''(B_s)\frac{ds}{\sqrt{2\pi s}}
\end{equation}
for all $t\in [0,T]$ and $f\in C^2({\mathbb R})$ satisfying the condition
\begin{equation}\label{sec6-eq6.2000}
|f(x)|,|f'(x)|,|f''(x)|\leq Ce^{\beta {x^2}},\quad x\in {\mathbb R}
\end{equation}
with $0\leq \beta<\frac{\sqrt{\pi}}{4\sqrt{T}}$.
\begin{proposition}\label{lem6.1}
Let $f\in C^1({\mathbb R})$. We have
\begin{equation}\label{sec6-eq6.3}
[f(B),B]^{(TQ)}_t=\int_0^tf'(B_s)\frac{ds}{\sqrt{2\pi s}}
\end{equation}
and in particular, we have
$$
[B,B]^{(TQ)}_t =\sqrt{\frac{2t}{\pi}}
$$
for all $t\geq 0$.
\end{proposition}
\begin{proof}
The proof is similar to Proposition~\ref{prop4.1}. It is enough to show that, for each $t\geq 0$
\begin{equation}\label{sec6-eq6.4}
\left\|B^\varepsilon_t-\sqrt{\frac{2t}{\pi}}
\right\|_{L^2}^2 =O(\varepsilon^\alpha)
\end{equation}
with some $\alpha>0$, as $\varepsilon\to 0$, by Lemma~\ref{Grad-Nourdin}, where
$$
B^\varepsilon_t=\frac1{\sqrt{\varepsilon}} \int_0^t(B_{s+\varepsilon}-B_s)^2d\sqrt{s}.
$$
We have
$$
E\left|B^\varepsilon_t-\sqrt{\frac{2t}{\pi}}\right|^2 =\frac{1}{\varepsilon}\int_0^t\int_0^t
A_\varepsilon(s,r)d\sqrt{s}d\sqrt{r}
$$
for $t\geq 0$ and $\varepsilon>0$, where
\begin{align*}
A_\varepsilon(s,r):&=E\left(
(B_{s+\varepsilon}-B_s)^2
-\sqrt{\frac{2\varepsilon}{\pi}}\right)\left(
(B_{r+\varepsilon}-B_r)^2
-\sqrt{\frac{2\varepsilon}{\pi}}\right)\\
&=E(B_{s+\varepsilon}-B_s)^2(B_{r+\varepsilon}-B_r)^2
+\frac{2\varepsilon}{\pi}\\
&\qquad-\sqrt{\frac{2\varepsilon}{\pi}}
E\left((B_{s+\varepsilon}-B_s)^2+(B_{r+\varepsilon}-B_r)^2\right).
\end{align*}
Defined the function $\phi_s:\;{\mathbb R}_{+}\to {\mathbb R}_{+}$ by
$$
\phi_s(x) =\frac1{\sqrt{2\pi}}\left(\sqrt{2(s+x)}
-2\sqrt{2s+x}+\sqrt{2s}\right)
$$
for every $s>0$. Then, we have
\begin{align*}
E\left[(B_{s+\varepsilon}-B_s)^2\right]&= \frac1{\sqrt{2\pi}}\left(\sqrt{2(s+\varepsilon)}
-2\sqrt{2s+\varepsilon}+2\sqrt{\varepsilon}+\sqrt{2s}\right) =\phi_s(\varepsilon)+\sqrt{\frac{2\varepsilon}{\pi}}
\end{align*}
for all $s>0$. Noting that
\begin{align*}
E[(&B_{s+\varepsilon}-B_s)^2(B_{r+\varepsilon}-B_r)^2]\\
&=E\left[(B_{s+\varepsilon}-B_s)^2\right]
E\left[(B_{r+\varepsilon}-B_r)^2\right] +2\left(E\left[(B_{s+\varepsilon}-B_s)
(B_{r+\varepsilon}-B_r)\right]\right)^{2}
\end{align*}
for all $r,s\geq 0$ and $\varepsilon>0$, we get
\begin{align*}
A_\varepsilon(s,r)&=\phi_s(\varepsilon)\phi_r(\varepsilon) +2(\mu_{s,r})^2
\end{align*}
where $\mu_{s,r}:=E\left[(B_{s+\varepsilon}-B_s)
(B_{r+\varepsilon}-B_r)\right]$. Now, let us estimate the function
$$
\phi_s(\varepsilon) =\frac1{\sqrt{2\pi}}\left(\sqrt{2(s+\varepsilon)}
-2\sqrt{2s+\varepsilon}+\sqrt{2s}\right).
$$
Clearly, one can see that
$$
\lim_{x\to 0}\frac{1-2\sqrt{1-x/2}+\sqrt{1-x}}{x^2}=-\frac1{16}
$$
and the continuity of the function $x\mapsto 1-2\sqrt{1-x/2}+\sqrt{1-x}$ implies that
\begin{align*}
|\phi_s(\varepsilon)|&= \sqrt{2(s+\varepsilon)}|1-2\sqrt{1-x/2}+\sqrt{1-x}|\leq C\frac{\varepsilon^2}{(s+\varepsilon)^{3/2}}\leq C\frac{\varepsilon^{\frac12+\beta}}{(s+\varepsilon)^\beta}
\end{align*}
with $x=\frac{\varepsilon}{s+\varepsilon}$ and $0<\beta<\frac12$, which gives
\begin{align*}
\frac{1}{\varepsilon}&\int_0^t\int_0^t  |\phi_s(\varepsilon)\phi_r(\varepsilon)|d\sqrt{s}d\sqrt{r}\leq Ct^{1-2\beta}\varepsilon^{2\beta}.
\end{align*}
It follows from Lemma~\ref{lem2.4} that there is a
constant $\alpha>0$ such that
\begin{align*}
\lim_{\varepsilon\downarrow 0}\frac{1}{
\varepsilon^{1+\alpha}}\int_0^t\int_0^tA_\varepsilon(s,r) d\sqrt{s}d\sqrt{r}=0
\end{align*}
for all $t>0$, which gives the desired estimate
$$
\left\|B^\varepsilon_t-\sqrt{\frac{2t}{\pi}} \right\|_{L^2}^2=O\left(
\varepsilon^\alpha\right)\qquad (\varepsilon\to 0)
$$
for each $t\geq 0$ and some $\alpha>0$. This completes the proof.
\end{proof}
Consider the decomposition
\begin{equation}\label{sec6-eq6.5}
\begin{split}
I_\varepsilon^2(f,x,t)&=\frac1{\sqrt{\varepsilon}}\int_0^t f(B_{s+\varepsilon}) (B_{s+\varepsilon}-B_s)\frac{ds}{2\sqrt{s}}\\
&\hspace{2cm}-\frac1{\sqrt{\varepsilon}}\int_0^t f(B_s)(B_{s+\varepsilon}-B_s)\frac{ds}{2\sqrt{s}}\\
&\equiv I_\varepsilon^{2,+}(f,x,t)-I_\varepsilon^{2,-}(f,x,t)
\end{split}
\end{equation}
for $\varepsilon>0$, and by estimating the two terms in the right hand side above in $L^2(\Omega)$ one can structure the next Banach space:
$$
{\mathscr H}_{\ast}=\{f\,:\,{\text { Borel functions on ${\mathbb R}$ such that $\|f\|_{{\mathscr H}_{\ast}}<\infty$}}\},
$$
where
\begin{align*}
\|f\|_{{\mathscr H}_{\ast}}^2:&=\frac{1}{\sqrt[4]{4\pi}}\int_0^T\int_{\mathbb
R}|f(z)|^2e^{-\frac{z^2\sqrt{\pi}}{2\sqrt{s}}} \frac{dzds}{s^{3/4}}\equiv \int_0^TE|f(B_s)|^2\frac{ds}{2\sqrt{s}}.
\end{align*}
Clearly, ${\mathscr H}_\ast=L^2({\mathbb R},\mu(dz))$ with
$$
\mu(dz)=\left(
\frac{1}{\sqrt[4]{4\pi}}\int_0^T e^{-\frac{z^2\sqrt{\pi}}{2\sqrt{s}}} \frac{ds}{s^{3/4}}\right)dz,
$$
and ${\mathscr H}_\ast$ includes all functions $f$ satisfying the condition
$$
|f(x)|\leq Ce^{\beta {x^2}},\quad x\in {\mathbb R}
$$
with $0\leq \beta<\frac{\sqrt{\pi}}{4\sqrt{T}}$. In a same way proving Theorem~\ref{th3.1} and by smooth approximation one can introduce the following result.
\begin{theorem}\label{th6.1}
The PQC $[f(B),B]^{(TQ)}$ exists and
\begin{align}\label{sec6-eq6.6}
E\left|[f(B),B]^{(TQ)}_t\right|^2\leq C \|f\|_{{\mathscr H}_{\ast}}^2
\end{align}
for all $f\in {\mathscr H}_{\ast}$ and $t\in [0,T]$. Moreover, if $F$ is an absolutely continuous function such that
$$
|F(x)|,|F'(x)|\leq Ce^{\beta {x^2}},\quad x\in {\mathbb R}
$$
with $0\leq \beta<\frac{\sqrt{\pi}}{4\sqrt{T}}$, then the following It\^o type formula holds:
\begin{equation}
F(B_t)=F(0)+\int_0^tF'(B_s)\delta B_s+\frac1{2\sqrt{2}}[F'(B),B]^{(TQ)}_t
\end{equation}
for all $t\in [0,T]$.
\end{theorem}

Recall that Russo and Tudor~\cite{Russo-Tudor} has showed that
$B=\{B_t=u(t,\cdot),t\geq 0\}$ admits a local time $L(t,a)\in L^2(\lambda\times P)$ such that $(a,t)\mapsto L(a,t)$ is jointly continuous, where $\lambda$ denotes Lebesgue measure, since $B=\{B_t=u(t,\cdot),t\geq 0\}$ is a bi-fractional Brownian motion for every $x\in {\mathbb R}$. Define the weighted local time ${\mathscr L}$ of $B=\{B_t=u(t,\cdot),t\geq 0\}$ by
\begin{align*}
{\mathscr L}(x,t)&=\int_0^t\frac{1}{2\sqrt{\pi s}}d_sL(s,x)\equiv \int_0^t\delta(B_s-x)\frac{ds}{2\sqrt{\pi s}}
\end{align*}
for $t\geq 0$ and $x\in {\mathbb R}$, where $\delta$ is the Dirac delta function. Then, the occupation formula
\begin{equation}\label{sec6-eq6.8}
\int_0^t\psi(B_s,s)\frac{ds}{2\sqrt{\pi s}}=\int_{\mathbb R}da\int_0^t\psi(a,s){\mathscr L}(a,ds)
\end{equation}
holds for every continuous and bounded function $\psi:{\mathbb
R}\times {\mathbb R}_{+}\rightarrow {\mathbb R}$ and any $x\geq 0$. As in Section~\ref{sec4-2}, we can show that the integral
$$
\int_{\mathbb R}f_\triangle(x){\mathscr
L}(dx,t):=\sum_jf_j\left[{\mathscr L}(a_j,t)-{\mathscr
L}(a_{j-1},t)\right].
$$
is well-defined and
\begin{equation}\label{sec6-eq6.9}
\int_{\mathbb R}f_{\Delta}(x)\mathscr{L}(dx,t)=
-\frac1{\sqrt{2}}[f_\triangle(B),B]^{(TQ)}_t
\end{equation}
almost surely, for all $f_\triangle=\sum_jf_j1_{(a_{j-1},a_j]}\in {\mathscr E}$. By the denseness of ${\mathscr
E}$ in ${\mathscr H}_{\ast}$ one can define
$$
\int_{\mathbb R}f(x){\mathscr L}(dx,t):=\lim_{n\to
\infty}\int_{\mathbb R}f_{\triangle,n}(x){\mathscr L}(dx,t)
$$
for any $f\in {\mathscr H}_{\ast}$, where $\{f_{\triangle,n}\}\subset {\mathscr E}$ and
$$
\lim_{n\to \infty}f_{\triangle,n}=f
$$
in ${\mathscr H}$. Moreover, the Bouleau-Yor type formula
\begin{equation}\label{sec6-eq6.10}
[f(B),B]^{(TQ)}_t=-\sqrt{2}\int_{\mathbb R}f(x)\mathscr{L}(dx,t)
\end{equation}
holds, almost surely, for all $f\in {\mathscr H}_{\ast}$.
\begin{corollary}[Tanaka formula]
For any $x\in {\mathbb R}$ we have
\begin{align*}
|B_t-x|=|x|+\int_0^t{\rm sign}(B_s-x)\delta B_s+{\mathscr L}(x,t).
\end{align*}
\end{corollary}

\end{document}